\documentclass[11pt, reqno,twoside]{article}
\usepackage{amssymb}
\usepackage{amsmath}
\numberwithin{equation}{section}
\usepackage{amsthm}
\usepackage{graphicx}
\usepackage{ifthen}
\usepackage[pagebackref,colorlinks,citecolor=blue,linkcolor=blue]{hyperref}
\usepackage{cite}
\usepackage{tikz-cd}
\usepackage{todonotes}
\usepackage{amsmath}
\usepackage{amsfonts}
\usepackage{amssymb}
\usepackage{amsmath,bbm,amssymb,amsxtra}
\usepackage{mathrsfs}
\usepackage{enumerate}
\usepackage{caption}
\allowdisplaybreaks
\usepackage{epsfig}
\usepackage{ae}
\usepackage{epstopdf}

\usepackage[left=2.6cm, right=2.6cm, top=2.6cm, bottom=2.6cm]{geometry}

\def\vint{\mathop{\mathchoice%
         {\setbox0\hbox{$\displaystyle\intop$}\kern 0.22\wd0%
          \vcenter{\hrule width 0.6\wd0}\kern -0.82\wd0}%
         {\setbox0\hbox{$\textstyle\intop$}\kern 0.2\wd0%
          \vcenter{\hrule width 0.6\wd0}\kern -0.8\wd0}%
         {\setbox0\hbox{$\scriptstyle\intop$}\kern 0.2\wd0%
          \vcenter{\hrule width 0.6\wd0}\kern -0.8\wd0}%
         {\setbox0\hbox{$\scriptscriptstyle\intop$}\kern 0.2\wd0%
          \vcenter{\hrule width 0.6\wd0}\kern -0.8\wd0}}%
         \mathopen{}\int}

\newcommand{\diam}{\text{\rm diam\,}}

\newcommand{\dist}{\text{\rm dist}}

\newtheorem{thm}{Theorem}[section]
\newtheorem{Thm}{Theorem}

\newtheorem{claim}{Claim}[section]
\newtheorem{rem}[thm]{Remark}
\newtheorem{rems}{Remarks}[section]
\newtheorem{cor}[thm]{Corollary}
\newtheorem{lem}[thm]{Lemma}
\newtheorem{prop}[thm]{Proposition}

\newtheorem{subclaim}{Subclaim}
\newtheorem{conj}[equation]{Conjecture}
\newtheorem{case}{Case}[section]
\newtheorem*{mysolution}{Solution}
\newtheorem{step}{Step}[section]
\theoremstyle{definition}
\newtheorem{defn}[thm]{Definition}

\newcounter {own}
\def\theown {\thesection       .\arabic{own}}

\newenvironment{pf}[1][]{%
 \vskip 3mm
 \noindent
 \ifthenelse{\equal{#1}{}}%
  {{\slshape Proof. }}%
  {{\slshape #1.} }%
 }%
{\qed\bigskip}

\newcounter{alphabet}

\makeatletter

\newcommand{\IR}{{\mathbb R}}

\newcommand{\RNum}[1]{\uppercase\expandafter{\romannumeral #1\relax}}

\def\be{\begin{equation}}
\def\ee{\end{equation}}

\newcommand{\ben}{\begin{enumerate}}
\newcommand{\een}{\end{enumerate}}

\newcommand{\blem}{\begin{lem}}
\newcommand{\elem}{\end{lem}}
\newcommand{\bthm}{\begin{thm}}
\newcommand{\ethm}{\end{thm}}
\newcommand{\bcor}{\begin{cor}}
\newcommand{\ecor}{\end{cor}}
\newcommand{\beg}{\begin{examp}}
\newcommand{\eeg}{\end{examp}}
\newcommand{\begs}{\begin{examples}}
\newcommand{\eegs}{\end{examples}}
\newcommand{\bdefe}{\begin{defn}}
\newcommand{\edefe}{\end{defn}}
\newcommand{\bprob}{\begin{prob}}
\newcommand{\eprob}{\end{prob}}
\newcommand{\bques}{\begin{ques}}
\newcommand{\eques}{\end{ques}}
\newcommand{\bei}{\begin{itemize}}
\newcommand{\eei}{\end{itemize}}
\newcommand{\bcl}{\begin{claim}}
\newcommand{\ecl}{\end{claim}}
\newcommand{\bscl}{\begin{subclaim}}
\newcommand{\escl}{\end{subclaim}}
\newcommand{\bca}{\begin{case}}
\newcommand{\eca}{\end{case}}
\newcommand{\bstep}{\begin{step}}
\newcommand{\estep}{\end{step}}
\newcommand{\bsol}{\begin{mysolution}}
\newcommand{\esol}{\end{mysolution}}
\newcommand{\bcon}{\begin{conj}}
\newcommand{\econ}{\end{conj}}
\newcommand{\bcons}{\begin{conjs}}
\newcommand{\econs}{\end{conjs}}
\newcommand{\bprop}{\begin{prop}}
\newcommand{\eprop}{\end{prop}}
\newcommand{\br}{\begin{rem}}
\newcommand{\er}{\end{rem}}
\newcommand{\brs}{\begin{rems}}
\newcommand{\ers}{\end{rems}}
\newcommand{\bo}{\begin{obser}}
\newcommand{\eo}{\end{obser}}
\newcommand{\bos}{\begin{obsers}}
\newcommand{\eos}{\end{obsers}}

\newcommand{\bpf}{\begin{pf}}
\newcommand{\epf}{\end{pf}}

\newcommand{\ba}{\begin{array}}
\newcommand{\ea}{\end{array}}
\newcommand{\beq}{\begin{eqnarray}}
\newcommand{\beqq}{\begin{eqnarray*}}
\newcommand{\eeq}{\end{eqnarray}}
\newcommand{\eeqq}{\end{eqnarray*}}

\begin{document}
\title{\Large\bf Quantitative correspondence between quasi-symmetric mappings on complete metric spaces and rough quasi-isometric mappings on their hyperbolic fillings
 \footnotetext{\hspace{-0.35cm}
 $2020$ {\it Mathematics Subject Classfication}: Primary: 30L10; Secondary: 30L05, 51M10.
 \endgraf{{\it Key words and phases}: quasi-symmetric mapping, rough quasi-isometric mapping, hyperbolic filling, metric measure space }
}
}
\author{Manzi Huang, Xiantao Wang, Zhuang Wang, Zhihao Xu}
\date{ }
\maketitle

\begin{abstract}
	In this paper, we establish a quantitative correspondence between power quasi-symmetric mappings on complete metric spaces and rough quasi-isometric mappings on their hyperbolic fillings. In particular, we prove that the exponents in the power quasi-symmetric mappings coincide with the coefficients in the rough quasi-isometric mappings. This shows that the obtained correspondence is both sharp and consistent. In this way, we generalize the corresponding result by Bj\"orn, Bj\"orn, Gill, and Shanmugalingam (J. Reine Angew. Math., 2017) from the setting of rooted trees to that of hyperbolic fillings.
\end{abstract}

\section{Introduction}\label{sec-1}
A construction termed {\it hyperbolic filling} has been widely used in the study of geometric group theory and analysis on metric measure spaces. It provides a method for studying the large-scale geometry of a metric space by embedding it into a Gromov hyperbolic space. This technique is particularly useful for understanding the quasi-isometric properties of hyperbolic groups and their boundaries; see for example \cite{BHK,BSC,Elek,Esk98,Gr93,Ka-B,Kle-Le98}, and references therein.
In the context of analysis on metric measure spaces, hyperbolic fillings have played a quite useful role in understanding uniform domains \cite{BHK,BuC,ZPR} and in studying various function spaces, which include Sobolev spaces \cite{BS}, Besov spaces \cite{BP03, BBS, BBGS, BC, ST17, KW}, Triebel-Lizorkin spaces \cite{ST17, BST} etc.

Let $(Z, d_Z)$ be a complete metric space. In the following, we always assume that all metric spaces considered herein contain at least three points.
The construction of hyperbolic fillings of $Z$ has been considered in, e.g., \cite{BST, BuSc, BC, BuC, Jordi, Map, ST17, So}. If in addition $Z$ is compact, some slightly different constructions were given in, e.g., \cite{BBS, BS, BP03, Ca13, HX, Li1, Sh22}. Similar constructions termed as, e.g.,  hyperbolic cones, were also discussed in \cite{BSC, Gr87, HXu, Ib11, Ib14, TrV}. In this paper, since the compactness of $Z$ is not required, we adopt the construction given in \cite{BC,BuC}, which is inspired by a construction due to Buyalo and Schroeder \cite[Chapter 6]{BuSc}; see Section~\ref{sec-3} for details.

In 2017, Bj\"{o}rn, Bj\"{o}rn, Gill, and Shanmugalingam \cite{BBGS} investigated the  correspondence between rough quasi-isometric mappings on rooted trees and power quasi-symmetric mappings on their boundaries. Specifically, they proved that every power quasi-symmetric mapping between the boundaries of two rooted trees admits a rough quasi-isometric extension to the trees themselves. Conversely, every rough quasi-isometric mapping between two rooted trees induces  a power quasi-symmetric mapping between their boundaries. In both directions, sharp estimates for the involved parameters were established. See \cite[Theorems~8.2 and 9.9]{BBGS}. It is worth noting that every rooted tree can be viewed as a hyperbolic filling of its boundary (cf. \cite[Theorem~7.1]{BBS}). Morevover, the boundary of every rooted tree is a Cantor-type set, which is compact and uniformly perfect.

For general metric spaces and their hyperbolic fillings, the
correspondence between quasi-symmetric mappings on metric spaces and
rough quasi-isometric mappings on their hyperbolic fillings has also attracted much attention. In 2000, Bonk and Schramm  \cite[Theorem~7.4]{BSC} proved that every power quasi-symmetric mapping $f:$ $Z\rightarrow W$ of bounded metric spaces can be extended to a rough quasi-isometric mapping between their hyperbolic cones $F:$ ${\rm Con}(Z)\rightarrow {\rm Con}(W)$. In 2007, Buyalo and Schroeder  \cite[Theorem~7.2.1]{BuSc} obtained that for each quasi-symmetric mapping $f:$ $Z\rightarrow W$ on uniformly perfect and complete metric spaces, there is a rough quasi-isometric mapping on their hyperbolic fillings $F:$ $X\rightarrow X^\prime$ which induces $f$. This means that the boundary mapping $\partial_\infty F$ of $F$ coincides with $f$ on $Z$. The precise definition of the boundary mapping $\partial_\infty F:$ $\partial_G X \to \partial_G X^\prime$ will be presented in Section~\ref{sec-4}; see also \cite[Section 6]{BSC}. Note that, in \cite{BuSc}, the authors used a different name for hyperbolic fillings, that is, hyperbolic approximations.
By \cite[Theorem~11.3]{H}, we see that every quasi-symmetric mapping is power quasi-symmetric provided that the underlying space is uniformly perfect.
As a generalization of \cite[Theorem~7.2.1]{BuSc}, Jordi \cite[Theorem~1]{Jordi} and Mart\'inez-P\'erez \cite[Theorem~1.7]{Map} independently showed that every power quasi-symmetric mapping $f: Z\rightarrow W$ of complete metric spaces admits a rough quasi-isometric extension of their hyperbolic fillings $F: X\rightarrow X^\prime$ which induces $f$. On the other hand, since all hyperbolic fillings mentioned above are Gromov hyperbolic geodesic spaces, it follows from \cite[Theorem~5.2.17]{BuSc} that every rough quasi-isometric mapping between hyperbolic fillings induces a power quasi-symmetric mapping between their boundaries.

However, the sharp estimates for the involved parameters are lacking of consideration in the aforementioned results on hyperbolic fillings. Motivated by sharp estimates established by Bj\"{o}rn, Bj\"{o}rn, Gill, and Shanmugalingam \cite{BBGS} in the context of rooted trees, in this paper, we shall extend their result, i.e., Theorems~8.2 and 9.9 in \cite{BBGS}, from  rooted trees to  hyperbolic fillings of general metric spaces. In particular, we seek to clarify the sharp estimates for the main parameters associated with the corresponding mappings.

Returning to the setting of rooted trees,  it is known that a natural partial order exists on such structures.  By mapping each vertex of one tree  to a certain common ``largest ancestor" (which is unique due to the partial order) in another rooted tree, Bj\"orn, Bj\"orn, Gill, and Shanmugalingam provided a method to extend  power quasi-symmetric mappings between the boundaries of two rooted trees. Importantly, the resulting mapping preserves this partial order. By exploiting this order-preserving property, they were able to derive the precise relations between the parameters. In the same paper, applying this correspondence, they established an embedding result for Besov spaces on the boundaries of rooted regular trees \cite[Theorem~8.3]{BBGS}.

In the context of hyperbolic fillings of general metric spaces, the situation becomes more delicate. Such a well-behaved partial order no longer exists, so the notion of a ``largest ancestor" may not be unique. To overcome this obstacle, we develop an alternative extension method that uses infinite hyperbolic cones as a bridge to connect hyperbolic fillings. This approach offers a clear geometric interpretation of the correspondence of vertices in the hyperbolic filling,
thereby enabling  precise computations of the relations between the parameters.

To state our result, some preparation is needed.
For a complete metric space $(Z, d_Z)$, denote by $X_Z$ its hyperbolic filling with construction parameters $\alpha>1$ and $\tau>1$, and by $V_Z$ the vertex set of $X_Z$; see Section~\ref{sec-3} for details. According to \cite[Propositions 5.9]{BuC}, it is known that, under a certain constraint on  $\alpha$ and $\tau$, for example, $\tau>\max\left\{3, \alpha/(\alpha-1)\right\}$, $X_Z$ is Gromov hyperbolic. Moreover, it follows from \cite[Lemma 5.11 and Proposition 5.13]{BuC} that there exists a unique point $\omega\in \partial_G  X_Z$ such that the boundary $\partial_\omega X_Z=\partial_G X_Z\setminus \{\omega\}$ is canonically identified with $Z$ such that $d_Z$ performs as a  visual metric on $\partial_\omega X_Z$, where $\partial_G X_Z$ denotes the Gromov boundary of $X_Z$. Based on this fact, in the rest of this paper, we will not distinguish between $\partial_\omega X_Z$ and $Z$. Our main result reads  as follows.

\begin{thm}\label{intro-thm-emb}
	For complete metric spaces $(Z, d_Z)$ and $(W, d_W)$, suppose that $X_Z$ and $X_W$ are the hyperbolic fillings of $(Z, d_Z)$ and $(W, d_W)$ associated to parameters $\alpha$ and $\tau$, respectively, where $\alpha>1$ and $\tau>\max\left\{3, {\alpha}/({\alpha-1})\right\}$.
	Let $\omega \in \partial_G X_Z$ and $\omega^\prime \in \partial_G X_W$ be points for which $\partial_\omega X_Z$ is identified with $Z$ and $\partial_{\omega^\prime} X_W$ is identified with $W$.
	Let $\theta\geq 1$ and $\lambda\geq 1$ be constants. Then the following statements are true.
	\begin{enumerate}
		\item[$(i)$]
		Suppose that $f:$ $Z\rightarrow W$ is a $(\theta, \lambda)$-power quasi-symmetric mapping.
		Then there is a $(\theta, \Lambda)$-rough quasi-isometric extension $F:$ $X_Z\rightarrow X_W$ which induces $f$, i.e., the boundary mapping $\partial_\infty F=f$ on $Z$, where $\Lambda=\Lambda(\theta, \lambda, \alpha,\tau)$. Moreover, $F$ maps the vertex set $V_Z$ into the vertex set $V_W$.
		\item[$(ii)$]
		Suppose that  $F:$ $X_Z\rightarrow X_W$  is a $(\theta, \lambda)$-rough quasi-isometric mapping such that its boundary mapping $\partial_\infty F$  maps $\omega$ to $\omega^\prime$.
		Then $\partial_\infty F:$ $Z\rightarrow W$ is a $(\theta, \Lambda^\prime)$-power quasi-symmetric mapping, where $\Lambda^\prime=\Lambda^\prime(\theta, \lambda, \alpha, \tau)$.
	\end{enumerate}
	Here, the notation $\Lambda(\theta, \lambda, \alpha,\tau)$ $($resp. $\Lambda^\prime(\theta, \lambda, \alpha, \tau)$$)$ indicates that the constant $\Lambda$ $($resp. $\Lambda^\prime$$)$ depends only on the given parameters $\theta, \lambda, \alpha$ and $\tau$.
\end{thm}

\begin{rem}\rm
	\begin{enumerate}
		\item [$(1)$] Theorem~\ref{intro-thm-emb} is a direct consequence of Theorems~\ref{thm-emb}, \ref{prop-boundary}, \ref{prop-trun-X}, and \ref{thm-limit}. In fact, we prove more than Theorem~\ref{intro-thm-emb} in this paper.
		
		\item [$(2)$] Theorem~\ref{intro-thm-emb} can be viewed as a quantitative version of the correspondence between quasi-symmetric mappings on complete metric spaces and rough quasi-isometric mappings on their hyperbolic fillings. Moreover, the main parameters associated with the involved mappings are sharp and consistent. This is because the statement $(i)$ in Theorem~\ref{intro-thm-emb} shows that a power quasi-symmetric mapping can be extended to a rough quasi-isometric mapping whose coefficient equals the original exponent, and, conversely, the statement $(ii)$ in Theorem~\ref{intro-thm-emb} illustrates that a rough quasi-isometric mapping induces a power quasi-symmetric mapping whose exponent coincides with the original coefficient. This reciprocal relationship confirms that both the exponents of power quasi-symmetric mappings and the coefficients of rough quasi-isometric mappings are sharp.
	\end{enumerate}
\end{rem}

Incidentally, the precise quantitative correspondence in Theorem~\ref{intro-thm-emb} plays a vital role in our forthcoming work \cite{H-W-W-X-prep}, where it is used to obtain a sharp embedding result
induced by power quasi-symmetric mappings
 for Besov spaces on Ahlfors regular metric spaces.

The paper is organized as follows. In Section~\ref{sec-2}, some necessary terminologies are introduced. In Section~\ref{sec-3}, the concepts of Gromov hyperbolicity and Busemann function are introduced, and the constructions of hyperbolic fillings and infinite hyperbolic cones are presented. Several useful known results are recalled.
In Section~\ref{sec-4}, by using infinite hyperbolic cones as a bridge to connect the hyperbolic fillings, a different extension method is provided, see Theorem~\ref{thm-emb} and its proof. Also, it is proved that the obtained extension induces the original mapping on the boundary, see Theorem~\ref{prop-boundary}. Further, the extension obtained in Theorem~\ref{thm-emb} can be modified to satisfy the vertex-to-vertex property, see Theorem~\ref{prop-trun-X}. In fact,  Theorem~\ref{intro-thm-emb}$(i)$ is a special case of the combination of Theorems~\ref{thm-emb}, \ref{prop-boundary} and  \ref{prop-trun-X}. In Section~\ref{sec-5}, it is shown that the boundary mapping of a rough quasi-isometric mapping between hyperbolic fillings is power quasi-symmetric, see Theorem~\ref{thm-limit}. As a special case, Theorem~\ref{intro-thm-emb}$(ii)$ follows immediately.

\section{Rough quasi-isometric mappings and quasi-symmetric mappings }\label{sec-2}

Let $(X, d_X)$ be a metric space. The distance of sets $A$ and $B$ in $X$ is denoted by $\dist(A, B)$, i.e., $\dist(A, B)=\inf\{d_X(x,z):\; x\in A,\; z\in B\}$. The diameter of a set $A\subset X$ is denoted by $\diam A$, i.e., $\diam A=\sup\{d_X(x,z):\; x,\; z\in A\}$.

A set $A\subset X$ is called {\it $k$-cobounded} (in $X$) if there is a constant $k\geq 0$ such that $\dist(\{x\}, A)\leq k$ for any point $x\in X$. If $A$ is $k$-cobounded for some $k\geq0$,  we briefly say that $A$ is {\it cobounded}.

Let $f: X\rightarrow Y$ be a mapping (not necessary continuous) between metric spaces $(X, d_X)$ and $(Y, d_Y)$. Let $\alpha_1$, $\alpha_2$ and $\alpha$ be constants such that $\alpha_2\geq\alpha_1>0$ and $\alpha\geq1$. Suppose that $f(X)$ is $k$-cobounded in $Y$ for $k\geq 0$. If, in addition, for all $x, z\in X$,
$$
\alpha_1 d_X(x, z)-k\leq d_Y\left(f(x),f(z)\right)\leq \alpha_2d_X(x, z)+k,
$$
then $f$ is called an $(\alpha_1, \alpha_2, k)$-{\it rough quasi-isometric mapping}.

If
$$
\alpha^{-1}d_X(x,z)-k\leq d_Y(f(x),f(z))\leq\alpha d_X(x,z)+k,
$$
then $f$ is called an $(\alpha, k)$-{\it rough quasi-isometric mapping}, i.e., an $(\alpha^{-1}, \alpha, k)$-rough quasi-isometric mapping. For convenience, we call $\alpha$ a {\it coefficient} of $f$.

If
$$
\alpha d_X(x,z)-k\leq d_Y(f(x),f(z))\leq\alpha d_X(x,z)+k,
$$
then $f$ is called an $(\alpha, k)$-{\it rough similarity}.




Two mappings $f,g:X\rightarrow Y$ are {\it roughly equivalent}, written $f\simeq g$, if there exists a constant $C\geq 0$ such that $d_Y(f(x), g(x))\leq  C$ for every point $x\in X$.
A {\it rough inverse} of a rough quasi-isometric mapping $f:X\rightarrow Y$ is a rough quasi-isometric mapping $g:Y\rightarrow X$ such that $g\circ f\simeq \text{id}_{X}$ and $f\circ g \simeq \text{id}_{Y}$, where $\text{id}_{X}$ (resp. $\text{id}_{Y}$) denotes the identity mapping defined on $X$ (resp. $Y$).


\begin{lem}\label{lem-composition}
Let $X$, $Y$, and $Z$ be metric spaces. Suppose that $f: X\to Y$ is an $(\alpha_1, \alpha_2, k_1)$-rough quasi-isometric mapping with $\alpha_2\geq \alpha_1>0$ and $k_1\geq 0$, and $g: Y\to Z$ is an $(\alpha_3, \alpha_4, k_2)$-rough quasi-isometric mapping with $\alpha_4\geq \alpha_3>0$ and $k_2\geq 0$. Then $g\circ f: X\to Z$ is an $(\alpha_1\alpha_3, \alpha_2\alpha_4, k)$-rough quasi-isometric mapping with $k=\alpha_4(k_1+1)+2k_2+1$.
\end{lem}
 \begin{proof}
The assumptions of the lemma ensure that for any $x_1$, $x_2\in X$,
$$
\alpha_1d_X(x_1, x_2)-k_1\leq d_Y(f(x_1), f(x_2))\leq \alpha_2d_X(x_1, x_2)+k_1
$$
and
\begin{equation}\label{10-10-1}
\alpha_3 d_Y(f(x_1), f(x_2))-k_2\leq d_Z(g\circ f(x_1), g\circ f(x_2))\leq \alpha_4 d_Y(f(x_1), f(x_2))+k_2.
\end{equation}
Thus we have
$$
\alpha_1\alpha_3d_X(x_1, x_2)-k_3\leq d_Z(g\circ f(x_1), g\circ f(x_2))\leq \alpha_2\alpha_4d_X(x_1, x_2)+k_3,
$$ where
$k_3=k_1\max\{\alpha_3, \alpha_4\}+k_2=k_1\alpha_4+k_2$.

To show that $g\circ f(X)$ is cobounded in $Z$, let $z\in Z$. Since $g(Y)$ is $k_2$-cobounded in $Z$, there is $y\in Y$ such that
\begin{equation}\label{eq-gyz}
d_Z(g(y), z)\leq k_2+1.
\end{equation}
Also, since $f(X)$ is $k_1$-cobounded in $Y$, there is $x\in Y$ such that
\begin{equation}\label{eq-fxy}
d_Y(f(x), y)\leq k_1+1.
\end{equation}
It follows from \eqref{eq-gyz} and \eqref{eq-fxy}, together with \eqref{10-10-1}, that
\begin{align*}
d_Z(g\circ f(x), z)\leq& d_Z(g\circ f(x), g(y))+d_Z(g(y), z)\\
\leq& \alpha_4d_Y(f(x), y)+k_2+k_2+1\\
\leq& \alpha_4(k_1+1)+2k_2+1=: k_4,
\end{align*}
which shows that $g\circ f(X)$ is $k_4$-cobounded in $Z$. It is clear that $k_3\leq k_4$. Hence, by taking $k=\max\{k_3, k_4\}=k_4$, we know that $g\circ f: X\to Z$ is an $(\alpha_1\alpha_3, \alpha_2\alpha_4, k)$-rough quasi-isometric mapping, and hence, the lemma is proved.
\end{proof}

A {\it geodesic} $($resp. a {\it geodesic ray}, a {\it geodesic segment}$)$ in $X$ is an isometry $\gamma: I\rightarrow X$, where $I$ is $\mathbb{R}$ $($resp. $[0,+\infty)$, a closed segment in $\mathbb{R}$$)$. A {\it geodesic metric space} is a metric space $X$ such that for any points $x, y\in X$, there is a geodesic segment connecting $x$ and $y$. We denote any geodesic segment with endpoints $x, y$ by $[x,y]$.
If the geodesic segment connecting $x$ and $y$ is not unique, then we use $[x,y]$ to denote one of these geodesics.

\begin{defn}
	Let $(Z, d_Z)$ and $(W, d_W)$ be two metric spaces. A homeomorphism $f: Z\rightarrow W$ is {\it $\eta$-quasi-symmetric} if there exists a self-homeomorphism $\eta$ of $[0, +\infty)$ such that for all triples of points $x, y, z\in Z$,
	$$\frac{d_W(f(x), f(z))}{d_W(f(y), f(z))}\leq \eta\left(\frac{d_Z(x, z)}{d_Z(y, z)}\right).$$


	If, in addition, there are constants $\theta\geq 1$ and $\lambda\geq 1$ such that
	\begin{equation*}\label{eq-1.1}
		\eta(t)=
		\left\{\begin{array}{cl}
			\lambda t^{ 1/\theta}& \text{for} \;\; 0<t<1, \\
			\lambda t^{\theta}& \text{for} \;\; t\geq 1,
		\end{array}\right.
	\end{equation*}
	then $f$ is called a $(\theta, \lambda)$-{\it power quasi-symmetric mapping}. For convenience, we call $\theta$ an {\it exponent} of $f$.
\end{defn}


\section{Gromov hyperbolic spaces}\label{sec-3}

In this section, we give a brief introduction of Gromov hyperbolic spaces and Busemann functions, and provide the constructions of hyperbolic fillings and infinite hyperbolic cones adopted in this paper. For more details, we refer interested readers to \cite{BuC, BuSc, BSC, GH, V, ZPR}.

\subsection{Gromov hyperbolic spaces}\label{sec-3-1}
Assume that $(X,d_X)$ denotes a metric space.
Given a triple of points $x, y, o\in X$, the {\it Gromov product $(x|y)_{o}$ based at $o$} is defined as
$$
(x|y)_{o}=\frac{1}{2}\big(d_X(x,o)+d_X(y,o)-d_X(x,y)\big).
$$
Then for any $x, y, o, o^{\prime}\in X$, we have
\begin{equation}\label{eq-base-1}
\big|(x|y)_{o}-(x|y)_{o^{\prime}}\big|\leq d_X(o,o^{\prime}).
\end{equation}

Let $a\vee b$ $($resp. $a\wedge b$$)$ denote the maximum $($resp. the minimum$)$ of $a, b\in \overline{\mathbb{R}}=\mathbb{R}\cup \{\infty\}$. The space $X$ is called {\it $\delta$-hyperbolic} if
there is a constant $\delta\geq 0$ such that for all $x, y, z, o\in X$,
$$
(x|y)_{o}\geq(x|z)_{o}\wedge(z|y)_{o}-\delta.
$$
If $X$ is $\delta$-hyperbolic for some $0\leq\delta<\infty$, we sometimes briefly say that $X$ is {\it Gromov hyperbolic}.

\begin{defn}
Let $X$ be $\delta$-hyperbolic and $o\in X$. A sequence of points $\{x_{i}\}\subset X$ is said to {\it converge to infinity} if
$$
(x_{i}|x_{j})_{o}\rightarrow\infty \;\; \mbox{as}\;\; i, j\rightarrow\infty.
$$
Two sequences $\{x_{i}\}$ and $\{y_{i}\}$ that converge to infinity are said to be {\it equivalent} if
$$
(x_{i}|y_{i})_{o}\rightarrow\infty \;\; \mbox{as}\;\; i\rightarrow\infty.
$$
\end{defn}

This defines an equivalence relation for sequences in $X$ converging to infinity. The convergence of a sequence and the equivalence of two sequences do not depend on the choice of the basepoint $o$ because of \eqref{eq-base-1}.
The {\it Gromov boundary} $\partial_{G}X$ of $X$ is defined as the set of all equivalence classes of sequences converging to infinity. For a point  $\omega\in\partial_{G}X$ and a sequence $\{x_n\}$ converging to infinity, we say that {\it $\{x_n\}$ converges to $\omega$} and write $\{x_n\}\in\omega$ or $x_n\rightarrow\omega$ if $\{x_n\}$ belongs to the equivalence class of $\omega$.

Let $X$ be $\delta$-hyperbolic, and let  $y$, $o\in X$ and $\xi\in \partial_{G}X$. The Gromov product $(y|\xi)_o$ based at $o$ is defined as follows:
$$
(\xi|y)_o=(y|\xi)_o=\inf\left\{\liminf_{i\rightarrow\infty}(x_i|y)_o:\; \{x_i\}\in\xi\right\}.
$$
For $\zeta, \xi\in\partial_{G}X$, we define the Gromov  product
$$
(\zeta|\xi)_o=\inf\left\{\liminf_{i\rightarrow\infty}(x_i|y_i)_o:\; \{x_i\}\in\zeta\;\;\mbox{and}\;\; \{y_i\}\in\xi\right\}.
$$

By \cite[(3.4) in Section 2]{BSC} and \cite[Lemma 5.11]{V}, we see that there is a constant $C(\delta)>0$ such that for any $\zeta, \xi\in\partial_{G}X$, if $\{x_i\}\in \zeta$ and $\{y_i\}\in \xi$, then
\begin{equation}\label{eq-Gromov-limit}
(\zeta|\xi)_o-C(\delta)\leq \liminf_{i\rightarrow\infty}(x_i|y_i)_o\leq \limsup_{i\rightarrow\infty}(x_i|y_i)_o\leq (\zeta|\xi)_o+C(\delta).
\end{equation}

\subsection{Busemann functions}\label{sec-3-15}
Let $(X, d_X)$ be a $\delta$-hyperbolic geodesic space, and let $\gamma:[0, +\infty)\rightarrow X$ be a geodesic ray. For a point $\omega\in \partial_G X$, we say that {\it $\gamma$ belongs to $\omega$} or $\gamma\in \omega$ if $\gamma(n)\to\omega$ as $n\to +\infty$.
The {\it Busemann function} $b_\gamma:X\rightarrow\mathbb{R}$ associated to $\gamma$ is defined by
\begin{equation}
b_{\gamma}(x)=\lim_{t\rightarrow+\infty} \big(d_X(\gamma(t),x)-t\big).
\end{equation}

We define the set of all Busemann functions on $X$ as
$$
\mathcal{B}(X)=\{b_{\gamma}+s:\text{$\gamma$ is a geodesic ray in $X$}\;\;\mbox{and}\;\; s\in \mathbb{R}\}.
$$
 For such $b=b_{\gamma}+s\in \mathcal{B}(X)$, we say that $\omega\in\partial_G X$ is the {\it basepoint} of $b$ if $\gamma$ belongs to $\omega$.
By \cite[Lemma 2.5]{BuC}, for any geodesic rays $\gamma$ and $\gamma^\prime$ that both belong to $\omega$, there exists an $s\in\mathbb R$ depending only on $\gamma(0)$ and $\gamma^\prime(0)$ such that
\begin{equation}\label{eq-base-2}
|b_\gamma-b_{\gamma^\prime}-s|\leq C(\delta).
\end{equation}
Moreover, $s=0$ if $\gamma(0)=\gamma^\prime(0)$.

Fix $b\in \mathcal{B}(X)$ with the basepoint $\omega\in\partial_G X$. For any $x, y\in X$, the {\it Gromov product $(x|y)_b$ based at $b$} is defined by
$$
(x|y)_b=\frac{1}{2}(b(x)+b(y)-d_X(x,y)).
$$

A sequence $\{x_n\}$ {\it converges to infinity with respect to $\omega$} if
$$
(x_m|x_n)_b\rightarrow \infty \;\; \mbox{as}\;\; m, n\rightarrow \infty,
$$
and two sequences $\{x_n\}$ and $\{y_n\}$ are {\it equivalent with respect to $\omega$} if
$$
(x_n|y_n)_b\rightarrow \infty \;\; \mbox{as}\;\;n\rightarrow \infty.
$$

By \cite[Lemma 2.5]{BuC}, for a fixed basepoint $\omega\in \partial_G X$, these definitions do not depend on the choice of $b\in \mathcal{B}(X)$ with this basepoint.
The {\it Gromov boundary relative to $\omega$}, denoted by $\partial_{\omega}X$, is the set of all equivalence classes of sequences converging to infinity with respect to $\omega$.
 For $\zeta\in\partial_{\omega}X$ and a sequence $\{x_n\}$ that converges to infinity with respect to $\omega$, we say that {\it $\{x_n\}\in\zeta$ with respect to $\omega$}, if $\{x_n\}$ belongs to the equivalence class of $\zeta$.

The following result  is derived  from \cite[Lemma 2.4]{BuC} and \cite[Example 3.2.1]{BuSc}.


\begin{lem}\label{eq-bbbb}
Let $\omega\in \partial_G X$ and $o\in X$. Let $\gamma$ be a geodesic ray from $o$ to $\omega$ with $\gamma(0)=o$, and let $b_\gamma$ be the Busemann function associated to $\gamma$. Then there is a constant $\nu=\nu(\delta)$ such that for any $x, y\in X$,
$$
|(x|y)_{b_\gamma}-\big((x|y)_o-(x|\omega)_o-(y|\omega)_o\big)|\leq \nu.$$
\end{lem}
\begin{proof}
By \cite[Lemma 2.4]{BuC}, we know that
\begin{equation*}
|b_\gamma(x) - \beta_{\omega, o}(x) |\leq C(\delta),
\end{equation*}
where $\beta_{\omega, o}(x)=d_X(o,x)-2(\omega | x)_o$.
Moreover, by \cite[Example 3.2.1]{BuSc}, we have
\begin{equation*}
(x|y)_o-(x|\omega)_o-(y|\omega)_o=\frac{1}{2}\left(\beta_{\omega, o}(x)+\beta_{\omega, o}(y)-d_X(x, y)\right),
\end{equation*}
which shows that
$$
|(x|y)_{b_\gamma}-\big((x|y)_o-(x|\omega)_o-(y|\omega)_o\big)|\leq \nu
$$
with $\nu=2C(\delta)$. The proof of this lemma is complete.
\end{proof}

The following result establishes an identification between $\partial_G X\setminus\{\omega\}$ and $\partial_{\omega}X$.

\begin{Thm}[{\cite[Proposition 3.4.1]{BuSc}}]\label{prop-Gromov}
Let $\omega\in\partial_G X$. A sequence $\{x_n\}$ converges to infinity with respect to $\omega$ if and only if $\{x_n\}$ converges to a point $\xi\in \partial_G X\setminus\{\omega\}$. This correspondence defines a canonical identification of $\partial_\omega X$ and $\partial_G X\setminus\{\omega\}$.
\end{Thm}
According to Theorem~\ref{prop-Gromov}, we shall thus use $\partial_{\omega}X$ instead of $\partial_G X\setminus\{\omega\}$ throughout the rest of the paper.

For $\xi\in \partial_{\omega}X$ and $y\in X$, we define
$$
(\xi|y)_b=(y|\xi)_b=\inf\left\{\liminf_{i\rightarrow\infty}(x_i|y)_b:\; \{x_i\}\in\xi\right\}.
$$
For $\zeta, \xi\in\partial_{\omega}X$, the Gromov product $(\zeta|\xi)_b$ based at $b$ is defined by
$$
(\zeta|\xi)_b=\inf\left\{\liminf_{i\rightarrow\infty}(x_i|y_i)_b:\; \{x_i\}\in\zeta\;\;\mbox{and}\;\; \{y_i\}\in\xi\right\}.
$$

\begin{Thm}[{\cite[Lemma 3.2.4]{BuSc} or \cite[Lemma 2.7]{BuC}}]\label{lem-fun-b}
Let $X$ be $\delta$-hyperbolic, $\omega\in \partial_G X$, and let $b$ be a Busemann function based at $\omega$. Then the following statements hold.
 \ben

\item[$(1)$]
For any $\xi, \zeta\in\partial_\omega X$ and any $\{x_i\}\in\xi$, $\{y_i\}\in\zeta$, we have
\begin{equation}\label{Lim}
(\xi|\zeta)_b\leq \liminf_{i\rightarrow\infty}(x_i|y_i)_b\leq \limsup_{i\rightarrow\infty}(x_i|y_i)_b\leq (\xi|\zeta)_b+600\delta,
\end{equation}
and the same holds if we replace $\zeta$ with $x\in X$.

\item[$(2)$]\label{lem-fun-b-2}
For any $\xi, \zeta, \eta\in X\cup \partial_\omega X$, we have
\begin{equation}\label{3-ponits}
(\xi|\zeta)_b\geq (\xi|\eta)_b \wedge (\eta|\zeta)_b-600\delta.
\end{equation}
\een
\end{Thm}

For $\epsilon>0$ and $b\in \mathcal{B}(X)$ with the basepoint $\omega$, we define a function $d_{\epsilon, b}$ on $\partial_{\omega}X$ as follows: For any $\zeta, \xi\in\partial_{\omega}X$, define
\begin{equation*}
d_{\epsilon, b}(\zeta, \xi)=e^{-\epsilon(\zeta|\xi)_{b}}.
\end{equation*}
In general, $d_{\epsilon, b}$ does not define a metric.
A metric $d$ on $\partial_{\omega}X$ is called a {\it visual metric} $($based at $b$$)$ with the parameter $\epsilon$ if ${\rm id}_{\partial_{\omega}X}: (\partial_{\omega}X, d)\to (\partial_{\omega}X, d_{\epsilon, b})$
is biLipschitz. It follows from \cite[Proposition 3.3.3]{BuSc}  that  visual metrics on $\partial_{\omega}X$ exist when $\epsilon$ is small enough.
The visual metrics on $\partial_{\omega}X$ do not depend on the choice of $b\in \mathcal{B}(X)$ with the basepoint $\omega$ because of \eqref{eq-base-2}.

A mapping $\Psi: X\rightarrow Y$ between metric spaces $X$ and $Y$ is called {\it strongly $(c_1, c_2, d)$-power quasi-isometric mapping} with $c_2\geq c_1>0$ and $d\geq 0$ if for all quadruples $\{x, y, z, u\}$ in $X$ with $\langle x, y, z, u\rangle\geq 0$,
\begin{equation*}
c_1 \langle x, y, z, u\rangle-d\leq \langle \Psi(x), \Psi(y), \Psi(z), \Psi(u)\rangle\leq c_2 \langle x, y, z, u\rangle+d,
\end{equation*}
where
$$
\langle x, y, z, u\rangle=(x|y)_o+(z|u)_o-(x|z)_o-(y|u)_o
$$
for any chosen basepoint $o\in X$. Obviously, for any $x, y, z, u\in X$,
\begin{equation}\label{eq-strong}
\langle x, y, z, u\rangle=-\langle x, z, y, u\rangle.
\end{equation}

For $c\geq 1$ and $d_0\geq 0$,  \cite[Theorem 4.4.1]{BuSc} states that any $(c^{-1}, c, d_0)$-rough quasi-isometric mapping  between hyperbolic geodesic spaces is strongly $(c^{-1}, c,  d)$-power quasi-isometric, where $d$ depends only on $c$, $d_0$, and the hyperbolicity constants. An analogous argument to  the proof of \cite[Theorem 4.4.1]{BuSc} shows that the same conclusion holds  for $(c_1, c_2, d_0)$-rough quasi-isometric mappings with $c_2\geq c_1>0$ and $d_0\geq 0$. The precise statement is as follows. We omit its proof here.

\begin{lem}\label{thm-strong}
Suppose that $X$ and $Y$ are $\delta_X$- and $\delta_Y$-hyperbolic geodesic spaces with $\delta_X\geq 0$ and $\delta_Y\geq 0$, respectively.
	Let $\Psi: X\rightarrow Y$ be a $(c_1, c_2, d_0)$-rough quasi-isometric mapping. Then $\Psi$ is strongly $(c_1, c_2, d)$-power quasi-isometric, where $d=d(c_1, c_2, d_0, \delta_X, \delta_Y)$.
\end{lem}

Given a Gromov hyperbolic space $X$ and a point $o\in X$,  we extend the function $\langle \cdot, \cdot, \cdot, \cdot \rangle$ from $X$ to $X \cup \partial_G X$ as follows. For any distinct points \( x, y, z, u \in X \cup \partial_G X \), the quantity $\langle x, y, z, u \rangle$ is defined by
$$
\langle x, y, z, u\rangle=(x|y)_o+(z|u)_o-(x|z)_o-(y|u)_o
$$
for any chosen basepoint $o\in X$.

\begin{lem}\label{lem-langle}
Let $X$ be  a $\delta$-hyperbolic geodesic space and $o\in X$.  Let $\omega\in\partial_G X$, and let $\gamma$ be a geodesic ray from $o$ to $\omega$ with $\gamma(0)=o$.
Let $b$ be the Busemann function associated to $\gamma$ based at $\omega$.  Then the following statements hold.
 \ben
\item[$(i)$]
For any distinct points $x, y, z\in \partial_\omega X$ and for any sequences $\{x_n\}\in x$, $\{y_n\}\in y$, $\{z_n\}\in z$, $\{\omega_n\}\in\omega$ in $X$,
	\begin{align}\label{eq-lan}
	\langle x, y, z, \omega\rangle-C_1\leq & \liminf_{n\to\infty}\langle x_n, y_n, z_n, \omega_n\rangle \notag\\
	\leq& \limsup_{n\to\infty}\langle x_n, y_n, z_n, \omega_n\rangle\leq \langle x, y, z, \omega\rangle+C_1,
	\end{align}
	where $C_1=C_1(\delta)\geq 0$.
 \item[$(ii)$]
 For any distinct points $x, y, z\in X \cup \partial_\omega X$,
	\begin{equation}\label{eq-bboo}
		|(x|y)_b-(x|z)_b-\langle x, y, z, \omega\rangle|\leq C_2,
	\end{equation}
	where $C_2=C_2(\delta)\geq 0$.
\een
\end{lem}
\begin{proof}
First, we check the relation in \eqref{eq-lan}.
For this, let $x, y, z\in \partial_\omega X$ be such that $x, y, z, \omega$ are distinct, and let $\{x_n\}\in x$, $\{y_n\}\in y$, $\{z_n\}\in z$ and $\{\omega_n\}\in\omega$ in $X$. Then we know from \eqref{eq-Gromov-limit}  that  there exists $C=C(\delta)\geq 0$ such that
\begin{equation*}
(x|y)_o-C \leq \liminf_{n\rightarrow\infty}(x_n|y_n)_o\leq \limsup_{n\rightarrow\infty}(x_n|y_n)_o\leq (x|y)_o+C,
\end{equation*}
\begin{equation*}
(x|z)_o-C \leq \liminf_{n\rightarrow\infty}(x_n|z_n)_o\leq \limsup_{n\rightarrow\infty}(x_n|z_n)_o\leq (x|z)_o+C,
\end{equation*}
\begin{equation*}
(y|z)_o-C \leq \liminf_{n\rightarrow\infty}(y_n|z_n)_o\leq \limsup_{n\rightarrow\infty}(y_n|z_n)_o\leq (y|z)_o+C,
\end{equation*}
and
\begin{equation*}
(z|w)_o-C \leq \liminf_{n\rightarrow\infty}(z_n|w_n)_o\leq \limsup_{n\rightarrow\infty}(z_n|w_n)_o\leq (z|w)_o+C.
\end{equation*}

Since $x$, $y$, $z$, $\omega$ are four distinct points, we know that all $(x|y)_o$, $(x|z)_o$, $(y|\omega)_o$ and $(z|\omega)_o$ are finite.
Then we get
\begin{align*}
\langle x, y, z, \omega\rangle &\leq \liminf_{n\rightarrow\infty}\big((x_n|y_n)_o+(z_n|\omega_n)_o-(x_n|z_n)_o-(y_n|\omega_n)_o\big)+4C \\
&\leq \limsup_{n\rightarrow\infty}\big((x_n|y_n)_o+(z_n|\omega_n)_o-(x_n|z_n)_o-(y_n|\omega_n)_o\big)+4C
 \leq \langle x, y, z, \omega\rangle+8C.
\end{align*}
This shows that the relation in \eqref{eq-lan} is true by letting $C_1=4C$.

Second, we check the estimate in \eqref{eq-bboo}. For this, let $x$, $y$ and $z$ be distinct points in $X \cup \partial_\omega X$, and let
 $\{x_n\}\in x$, $\{y_n\}\in y$, $\{z_n\}\in z$ and $\{\omega_n\}\in\omega$ be sequences in $X$.
Then by Lemma~\ref{eq-bbbb},
\begin{equation}\label{eq-xn-yn}
\big|(x_n|y_n)_b-(x_n|z_n)_b -\big((x_n|y_n)_o-(x_n|z_n)_o-(y_n|\omega_n)_o+(z_n|\omega_n)_o\big)\big|\leq C^\prime,
\end{equation}
where $C^\prime=C^\prime(\delta)$.

Since Theorem~\ref{lem-fun-b}$(1)$ gives
\begin{align*}
\liminf_{n\rightarrow\infty}\big((x_n|y_n)_b-(x_n|z_n)_b\big)-600\delta
\leq
(x|y)_b-(x|z)_b
 \leq \limsup_{n\rightarrow\infty}\big((x_n|y_n)_b-(x_n|z_n)_b\big)+600\delta,
\end{align*}
we infer from the statement $(i)$ in the lemma and \eqref{eq-xn-yn} that
\begin{align*}
(x|y)_b-(x|z)_b &\leq \limsup_{n\rightarrow\infty}\big((x_n|y_n)_o-(x_n|z_n)_o-(y_n|\omega_n)_o+(z_n|\omega_n)_o\big)+C^\prime+600\delta\\
&\leq  \langle x, y, z, \omega\rangle+C_1+C^\prime+600\delta
\end{align*}
and
\begin{align*}
(x|y)_b-(x|z)_b &\geq \liminf_{n\rightarrow\infty}\big((x_n|y_n)_o-(x_n|z_n)_o-(y_n|\omega_n)_o+(z_n|\omega_n)_o\big)-C^\prime-600\delta\\
&\geq  \langle x, y, z, \omega\rangle-C_1-C^\prime -600\delta.
\end{align*}
Consequently,
\begin{align*}
\big| (x|y)_b-(x|z)_b - \langle x, y, z, \omega\rangle\big|\leq C_1+C^\prime+600\delta,
\end{align*}
which proves \eqref{eq-bboo} by letting $C_2=C_1+C^\prime+600\delta$.
\end{proof}

\subsection{Hyperbolic fillings}\label{sec-3-2}
Let $(Z, d_Z)$ be a metric space.   Denote by $B_Z(x, r)=\{y\in Z:\; d_Z(y, x)<r\}$
 the open ball of radius $r$ centered at $x$, and for $\tau>0$, let $\tau B_Z(x, r)=\{y\in Z:\; d_Z(y, x)<\tau r\}$.

Let us introduce the {\it hyperbolic filling} $X$ of $Z$ based on the construction given by Bulter \cite{BuC,BC}.
Assume that $\alpha>1$ and $\tau>1$ are constants. For each $n\in\mathbb Z$, we select a maximal $\alpha^{-n}$-separated subset $S_n$ of $Z$. The existence of such a set is guaranteed by a standard application of Zorn's lemma. Then for each $n\in\mathbb Z$, the balls $B_Z(z, \alpha^{-n})$ with $z\in S_n$ cover $Z$.
Let
$$
V=\bigcup_{n\in\mathbb Z}V_n,
$$
where $V_n=\{(z, n):\; z\in S_n\}.$  We call each element $(z,n)$ in $V_n$ a {\it vertex}.

To each vertex $v=(z, n)$, we associate the ball $B_Z(v)=B_Z(z, \alpha^{-n})$. We also define the {\it height
function} $h: V\rightarrow \mathbb Z$ by $h(z, n)=n$,
and the {\it projection} $\pi: V\rightarrow Z$ by $\pi(z, n)=z$.

Given two different vertices $v, w\in V$, we say that $w$ is a {\it neighbor} of $v$, denoted by $w\sim v$,
if
\begin{equation*}
|h(v)-h(w)|\leq 1 \;\;\text{ and }\;\;\tau B_Z(v)\cap \tau B_Z(w) \neq \emptyset.
\end{equation*}

Define the hyperbolic filling $X$ of $Z$ to be the graph formed by the vertex set $V$ together with the above neighbor relation (edges), and say that $X$ is the hyperbolic fillings of $Z$ associated to parameters $\alpha$ and $\tau$.
Also, we call $\alpha$ and $\tau$ the {\it construction parameters} of $X$, and require that they satisfy the following relation:
$$
\tau>\max\bigg\{3, \frac{\alpha}{\alpha-1}\bigg\}.
$$

As Butler pointed out in \cite{BuC} that the above constraint of $\tau$ is assumed to ensure that the hyperbolic filling $X$ is connected.


Edges between vertices of different heights are called {\it vertical}. A geodesic (or a geodesic ray, or a geodesic segment) is said {\it vertical} if it is a subset of a union of vertical edges.

We consider $X$ to be a metric graph, where the edges are unit intervals.
The graph distance between two points $x, y \in X$, denoted by $|x-y|$, is the length of the shortest curve connecting them.
It can be shown that $X$ is geodesic and $\delta$-hyperbolic for some $\delta=\delta(\alpha, \tau)>0$ (see \cite[Proposition 5.9]{BuC}).

 For any $x$, $y\in X$, let us recall that $[x, y]$ denotes a geodesic segment connecting $x$ and $y$.
For any $x_0, y_0\in [x, y]$, the inclusion $[x_0, y_0]\subset [x, y]$ means that $[x_0, y_0]$ is the geodesic subsegment of $[x, y]$ connecting $x_0$ and $y_0$.
Clearly, if  $v$, $w\in V$ with $v\sim w$, then $[v, w]$ is an edge in $X$ connecting $v$ and $w$.

%
%
%
%

For any edge $[v, w]$ in $X$, we extend the height function $h$ to $[v, w]$ by
$$
h(x)=th(w)+(1-t)h(v)
$$
for $x\in [v, w]$ with $|x - v|=t\in[0, 1]$.
Then the height function is extended to be a function $h$ on $X$,  which satisfies
$$
|h(x)-h(y)|\leq |x-y|
$$
for any $x, y\in X$ (cf. \cite[Section 5]{BuC}).

A {\it descending geodesic ray} (resp. an {\it ascending geodesic ray}) $\gamma : [0, +\infty)\rightarrow X$ is a vertical geodesic ray such that $h(\gamma(t))$ is strictly decreasing (resp. strictly increasing) as a function of $t$. Then for a descending geodesic ray $\gamma$, we know from the definition of vertical geodesics that for any $t\geq 0$,
$h(\gamma(t))=h(\gamma(0))-t$,
and for an ascending geodesic ray $\gamma$, we have
$h(\gamma(t))=h(\gamma(0))+t$
for any $t\geq 0$.
 Let $\overline{Z}$ be the completion of $Z$. Still, we use $d_Z$ to denote the extension of the metric on $Z$ to its completion.
A vertical geodesic $\gamma$ is {\it anchored} at a point $z\in\overline{Z}$ if for each vertex $v\in\gamma$, $z\in B_{\overline Z}\left(\pi(v), \frac{\tau}{3}\alpha^{-h(v)}\right)$.
When the point $z$ does not need to be referenced, we will just say that $\gamma$ is anchored. For $z\in \overline{Z}$, we know from \cite[Lemma 5.10]{BuC} that there exist an ascending geodesic ray and a descending ray in $X$ anchored at it.
By \cite[Lemma 5.11]{BuC}, there exists a point $\omega\in \partial_G X$ such that all anchored descending geodesic rays in $X$ belong to $\omega$.


\subsection{Infinite hyperbolic cones}\label{sec-3-3}

In this section, we introduce a class of Gromov hyperbolic spaces based on metric spaces $(Z, d_Z)$, called
{\it infinite hyperbolic cones}.
The construction of infinite hyperbolic cones was considered in \cite{BSC, HXu}.
For a metric space $(Z, d_Z)$, its infinite hyperbolic cone is defined as
\begin{equation*}
	\text{Con}_h(Z)=Z\times(0, +\infty),
\end{equation*}
and a metric $\rho_h:$ $\text{Con}_h(Z)\times\text{Con}_h(Z)\rightarrow[0, +\infty)$
is defined by the formula: For $p=(x,s)$ and $q=(y,t)\in\text{Con}_h(Z)$,
\begin{equation}\label{metric-cone}
	\rho_h(p, q)=2\log\frac{d_Z(x, y)+s\vee t}{\sqrt{st}}.
\end{equation}
The metric space $(\text{Con}_h(Z), \rho_h)$ is Gromov hyperbolic (cf. \cite[Section 2]{HXu}).
For a point $z\in Z$, we denote by $R_z$ the ray in $\text{Con}_h(Z)$ that ends at $z\in Z$, that is,
\beq\label{25-5-30}
	R_z=\{z\}\times(0, +\infty).
\eeq
Then $\text{Con}_h(Z)=\bigcup_{z\in Z}R_z$.

\begin{Thm}[{\cite[Theorem 1.1]{HXu}}]\label{thm-cone}
	Suppose that $f$: $(Z,d_Z)\to (W,d_W)$ is a $(\theta, \lambda)$-power quasi-symmetric mapping with $\theta\geq 1$ and $\lambda\geq 1$. Then there is a $(\theta, k)$-rough quasi-isometric mapping
$$\widehat{f}:  {\rm Con}_h(Z) \to {\rm Con}_h(W),$$
where $k=k(\theta, \lambda)$.
\end{Thm}

For the convenience of the readers, we briefly describe the process of constructing this mapping $\widehat{f}$ (see \cite[Section 3]{HXu} for details). Assume that $f$: $(Z,d_Z)\to (W,d_W)$ is a $(\theta, \lambda)$-power quasi-symmetric mapping with $\theta\geq 1$ and $\lambda\geq 1$.
For $z\in Z$,  let $\Phi_z: \IR\to \IR$ be the non-decreasing and continuous function constructed in \cite[(4.8)]{HXu},  which satisfies
\begin{equation}\label{mon-01}
	\lim_{t\to -\infty}\Phi_z(t)=-\infty,  \;\;  \lim_{t\to +\infty}\Phi_z(t)=+\infty \;\;  \text{and} \;\;  \Phi_z(\mathbb R)=\mathbb R.
\end{equation}
Also, $\Phi_z: \IR\to \IR$ is a $(\theta, \mu_0)$-rough quasi-isometric mapping with $\mu_0=\mu_0(\theta, \lambda)$ (see \cite[Lemma 4.4]{HXu}).

%

Based on the function $\Phi_z$, the $(\theta, k)$-rough quasi-isometric mapping
$
\widehat{f}: \text{Con}_h(Z)\to \text{Con}_h(W)
$ in Theorem~\ref{thm-cone} is defined as follows:
For $(z, t)\in \text{Con}_h(Z)$,
	\begin{equation}\label{eq-def-fx}
	\widehat{f}(z, t)=\left(f(z), 2^{-\Phi_z\left(-\log_2 t\right)}\right).
	\end{equation}
Observe that, $\widehat{f}$ maps the ray $R_z$ onto the ray $R_{f(z)}$ for each $z\in Z$.	

We end  this section with the following known result which will be used later on.

\begin{Thm}[{\cite[Lemma 4.7]{HXu}}]\label{lem-cone-2}
 For any $x\neq y\in Z$, if  $t\geq d_Z(x, y)$,  then
$$
\left|\Phi_{x}\left(-\log_2 t\right)- \Phi_{y}\left(-\log_2 t\right)\right|\leq \mu_1,
$$
where $\mu_1=\mu_1(\theta, \lambda)$.
\end{Thm}


\section{Rough quasi-isometric extension of  quasi-symmetric mappings}\label{sec-4}
The purpose of this section is to formulate and prove two results, i.e., Theorems~\ref{thm-emb} and \ref{prop-boundary} below, from which Theorem~\ref{intro-thm-emb}$(i)$ follows. Further, the extension constructed in Theorem~\ref{thm-emb} is modified to satisfy the vertex-to-vertex property as stated in Theorem~\ref{prop-trun-X}.
Before this, we first introduce a rough similarity between the infinite hyperbolic cone and a hyperbolic filling of complete metric space (see \cite[Theorem 3.4]{HXu}).

Let $(Z, d_Z)$ be a complete metric space, $X_Z$ a hyperbolic filling of $Z$ with parameters $\alpha_Z>1$ and $\tau_Z$ satisfying $\tau_Z>\max\Big\{3, \frac{\alpha_Z}{\alpha_Z-1}\Big\}$, and let $V_Z$ denote the vertex set of $X_Z$.
As mentioned in Subsection~\ref{sec-3-2}, there is a point in $\partial_G X_Z$ such that all anchored descending geodesic rays in $X_Z$ belong to it. For convenience, we use $\omega$ to denote this point.
Let $\partial_\omega X_Z$ be the Gromov boundary of $X_Z$ relative to $\omega$.

Define a mapping $\psi: Z\rightarrow \partial_\omega X_Z$ by setting $\psi(z)=\xi$, where $\xi$ is the equivalence class in $\partial_\omega X_Z$ defined by an ascending geodesic ray anchored at $z$.
 The following result shows that the mapping $\psi$ determines an identification of $Z$ with
$\partial_\omega X_Z$.

\begin{Thm}[{\cite[Proposition 5.13]{BuC}}]\label{prop-identi}
The mapping $\psi: Z\rightarrow \partial_\omega X_Z$ defines an identification of $Z$ with
$\partial_\omega X_Z$.
Under this identification, the metric $d_Z$ on $Z$ defines a visual metric on $\partial_\omega X_Z$ with the parameter $\epsilon=\log\alpha_Z$.
\end{Thm}

According to Theorem~\ref{prop-identi}, we know that for any pair of points $\xi$ and $\xi^\prime$ in $\partial_\omega X_Z$, there are points $z$ and $z^\prime \in Z$ such that $\xi=\psi(z)$ and $\xi^\prime=\psi(z^\prime)$. Let
\begin{equation}\label{d-visual}
	 d_\omega(\xi, \xi^\prime)=d_Z(z, z^\prime).
	\end{equation}
Then the metric $d_Z$ on $Z$ induces a visual metric $d_\omega$ on $\partial_\omega X_Z$ with parameter $\epsilon=\log\alpha_Z$.  More precisely, there exists $C=C(\alpha_Z,\tau_Z)\geq1$ such that
\begin{equation}\label{compare}
C^{-1}\alpha_Z^{-(\xi|\xi^\prime)_b} \leq d_\omega(\xi, \xi^\prime)\leq  C \alpha_Z^{-(\xi|\xi^\prime)_b},
\end{equation}
where $b:=b_\gamma$ is a Busemann function and $\gamma$ is an anchored descending geodesic ray which belongs to $\omega$.

 From now on, we equip $\partial_\omega X_Z$ with the metric $d_\omega$, and then, we can identify $(\partial_\omega X_Z,d_\omega)$ with $(Z, d_Z)$ via the homeomorphism $\psi: Z\rightarrow \partial_\omega X_Z$. Therefore, in the remainder of this section, we will not distinguish between $\partial_\omega X_Z$ and $Z$.

 For any $z\in Z$, let $\gamma_z:\mathbb R\rightarrow X_Z$ be a vertical geodesic anchored at $z$ such that
\begin{equation}\label{eq-6-24}
h(\gamma_z(t))=t
\end{equation}
for any $t\in\mathbb R$. The existence of $\gamma_z$ is guaranteed by the requirement of $\tau_Z>\max\left\{3, \frac{\alpha_Z}{\alpha_Z-1}\right\}$ (see \cite[Lemma 5.10]{BuC}).
Then $\gamma_z|_{[0, +\infty)}\in z$, and thus, by Theorem~\ref{prop-Gromov},
\begin{equation}\label{eq-6-24-1}
\{\gamma_z(n)\}_{n\in\mathbb N}\in  z  \ \text{ with respect to } \omega.
\end{equation}
Moreover, it follows from \eqref{eq-6-24} that for any $x\in \gamma_z$,
\begin{equation}\label{25-1-29-1}
\gamma_z(h(x))=x.
\end{equation}

	In general, the anchored vertical geodesic $\gamma_z$ may not be unique at $z\in Z$. We make a notational convention: {\it In the rest of this section, for every $z\in Z$, we fix a vertical geodesic anchored at it satisfying \eqref{eq-6-24} and \eqref{eq-6-24-1}, denoted by $\gamma_z$.}

Let us recall a mapping
	$
	\sigma: {\rm Con}_h(Z)\rightarrow X_Z
	$ (cf. \cite[Subsection 3.2]{HXu}): For any $(\xi, s)\in {\rm Con}_h(Z),$
	\begin{equation}\label{f-con}
		\sigma(\xi, s)=\gamma_{\xi}\left(-\frac{\log s}{\log\alpha_Z}\right).
	\end{equation}
Clearly, $\sigma$ maps each ray $R_{\xi}$ onto the vertical geodesic $\gamma_{\xi}$.
For any $(\xi, s)\in {\rm Con}_h(Z)$,
\begin{equation}\label{sigma-h}
h(\sigma(\xi, s))=-\frac{\log s}{\log\alpha_Z}.
\end{equation}
For $s_1, s_2\in(0, +\infty)$, if $s_1<s_2$, then $h(\sigma(\xi, s_1))> h(\sigma(\xi, s_2))$.

 In \cite{HXu}, the first and the fourth authors of the paper constructed rough similarities from ${\rm Con}_h(Z)$ to $X_Z$; see \cite[Theorem 3.4]{HXu}.
 Although the construction of the hyperbolic filling in \cite{HXu} is not the same as the one used in this paper, by the similar reasoning as in the proof of \cite[Theorem 3.4]{HXu}, we see that this result is still valid for the setting of this paper. 
For the sake of application, we state it as a proposition, the proof of which we omit for brevity.
\begin{prop}\label{thm-similar}
The mapping $\sigma: ({\rm Con}_h(Z),\rho_h)\rightarrow (X_Z, |\cdot|)$ is a $(1/\log\alpha_Z, C_\sigma)$-rough similarity with $C_\sigma=C_\sigma(\alpha_Z, \tau_Z)$.
\end{prop}

The following proposition ensures the existence of such a rough inverse of $\sigma$.

\begin{prop}\label{inverse-sigma}
There exists a rough inverse $\sigma^{-1}: (X_Z, |\cdot|) \to ({\rm Con}_h(Z),\rho_h)$ of $\sigma$ satisfying the following properties.
\begin{enumerate}
	\item[$(i)$] $\sigma^{-1}$ is a $(\log\alpha_Z, C_{\sigma^{-1}})$-rough similarity with $C_{\sigma^{-1}} = C_{\sigma^{-1}}(\alpha_Z, \tau_Z)$.
	\item[$(ii)$] For any $(z, t) \in  ({\rm Con}_h(Z),\rho_h)$,  if $\sigma^{-1}(\sigma(z, t)) = (z_0, s)$, then
	\begin{equation}\label{y-y0-t}
		\sigma(z_0, s) = \sigma(z, t) \ \ \ \text{ and } \ \ \ s=t.
	\end{equation}
	If in addition $\sigma(z, t)$ is a vertex in $X_Z$, then
	$$d_Z(z, z_0) \leq \frac{2}{3}\tau_Z t.$$
\end{enumerate}
\end{prop}
 \begin{proof}
 	Given that $\sigma: ({\rm Con}_h(Z),\rho_h)\rightarrow (X_Z, |\cdot|)$ is a $(1/\log\alpha_Z, C_\sigma)$-rough similarity, for any two points $(\xi_1, t_1), (\xi_2, t_2)\in {\rm Con}_h(Z)$,
 	\begin{equation}\label{rough-simi}
 		\frac{\rho_h\big((\xi_1, t_1), (\xi_2, t_2)\big)}{\log \alpha_Z} -C_\sigma\leq |\sigma(\xi_1, t_1)-\sigma(\xi_2, t_2)|\leq \frac{\rho_h\big((\xi_1, t_1), (\xi_2, t_2)\big)}{\log \alpha_Z} +C_\sigma.
 	\end{equation}
 	Moreover, the image of ${\rm Con}_h(Z)$ under $\sigma$ is $C_\sigma$-cobounded in $X_Z$. This implies that  for any $x\in X_Z$, there  exists a point $(\xi, t)\in {\rm Con}_h(Z)$ such that
 \begin{equation}\label{sigma-cobounded}
 	|x- \sigma(\xi, t)|\leq C_\sigma+1.
 \end{equation}

Now, we construct a rough inverse $\sigma^{-1}$ of $\sigma$ as follows. Let $x\in X_Z$.
If the set
$$\sigma^{-1}(\{x\}):=\{(z, t)\in {\rm Con}_h(Z) : \; \sigma(z, t)=x\}$$
 is not empty, then we choose a point $(z_0, t_0)\in \sigma^{-1}(\{x\})$ and define $\sigma^{-1}(x)=(z_0, t_0)$. If $\sigma^{-1}(\{x\})$ is empty, by \eqref{sigma-cobounded}, there exists a point $(z^\prime_0 , t^\prime_0)\in {\rm Con}_h(Z)$ such that
 $$|x- \sigma(z^\prime_0 , t^\prime_0)|\leq C_\sigma+1.$$
In this case, we define $\sigma^{-1}(x)=(z^\prime_0 , t^\prime_0)$. It is clear that $\sigma^{-1}$ is a mapping from $X_Z$ to ${\rm Con}_h(Z)$.

By the construction, we know that for any $x\in X_Z$,
$$|\sigma\circ \sigma^{-1}(x)- x|\leq C_\sigma+1,$$   which implies that $\sigma\circ \sigma^{-1} \simeq \text{id}_{X_Z}$.
On the other hand, for any $x\in \sigma({\rm Con}_h(Z))$, we know that
\begin{equation}\label{si-in-x}
\sigma\circ \sigma^{-1}(x)= x.
\end{equation}
Then $\sigma\circ \sigma^{-1}\circ \sigma(z, t)=\sigma(z, t)$ for any $(z, t)\in {\rm Con}_h(Z)$. It follows from \eqref{rough-simi} that
$$\frac{1}{\log{\alpha_Z}}  \rho_h\big(\sigma^{-1}\circ \sigma(z, t), (z, t)\big)-C_\sigma\leq   |\sigma\circ\sigma^{-1}\circ \sigma(z, t)-\sigma(z, t)|=0,$$
 and thus, $ \rho_h(\sigma^{-1}\circ \sigma(z, t), (z, t))\leq C_\sigma\log\alpha_Z$. Hence, $\sigma^{-1}\circ \sigma\simeq \text{id}_{{\rm Con}_h(Z)}$. Therefore, $\sigma^{-1}$ is a rough inverse of $\sigma$.

Next, we show that the mapping $\sigma^{-1}: X_Z \to {\rm Con}_h(Z)$
 satisfies the statements $(i)$.
For this, let $x_1$, $x_2\in X_Z$, and let  $(z_1, t_1)$, $(z_2, t_2)\in {\rm Con}_h(Z)$ be such that $(z_1, t_1)=\sigma^{-1}(x_1)$ and $(z_2, t_2)=\sigma^{-1}(x_2)$. Then by \eqref{rough-simi}, we obtain
\begin{align*}
(\log\alpha_Z) |\sigma(z_1, t_1)-\sigma(z_2, t_2)|- C_\sigma\log\alpha_Z
&\leq  \rho_h(\sigma^{-1}(x_1), \sigma^{-1}(x_2)) \\
&\leq (\log\alpha_Z) |\sigma(z_1, t_1)-\sigma(z_2, t_2)|+ C_\sigma\log\alpha_Z.
\end{align*}
Note that $|x_i-\sigma(z_i, t_i)|\leq C_\sigma+1$ for $i=1, 2$. Hence we get
\begin{align*}
(\log\alpha_Z)  |x_1-x_2|-(3C_\sigma+2)\log\alpha_Z
&\leq  \rho_h(\sigma^{-1}(x_1), \sigma^{-1}(x_2)) \\
&\leq  (\log\alpha_Z) |x_1-x_2|+ (3C_\sigma+2)\log\alpha_Z.
\end{align*}

Still, it remains to show that $\sigma^{-1}(X_Z)$ is cobounded in ${\rm Con}_h(Z)$.
To reach this goal, let $(z, t)\in {\rm Con}_h(Z)$ and let $y=\sigma(z, t)$. Set $\sigma^{-1}(y)=(z^\prime, t^\prime)$. Then by \eqref{si-in-x},
$$
\sigma(z^\prime, t^\prime)=\sigma\circ\sigma^{-1}(y)=y=\sigma(z, t),
$$
and thus, it follows from \eqref{rough-simi} that
$$
\rho_h((z, t), \sigma^{-1}(y))=\rho_h\big((z, t), (z^\prime, t^\prime)\big)\leq (\log\alpha_Z)|\sigma(z, t)-\sigma\circ\sigma^{-1}(y)| +C_\sigma\log\alpha_Z=C_\sigma\log\alpha_Z.
$$
This shows that $\sigma^{-1}(X_Z)$ is $(C_\sigma\log\alpha_Z)$-cobounded in ${\rm Con}_h(Z)$.

Consequently, $\sigma^{-1}: X_Z \to {\rm Con}_h(Z)$ is a  $(\log\alpha_Z, C_{\sigma^{-1}})$-rough similarity with $C_{\sigma^{-1}}=(3C_\sigma+2)\log\alpha_Z$, and thus, the statement $(i)$ holds.

In the following, we show that $\sigma^{-1}$ also satisfies the statement $(ii)$. Let $(z, t)\in {\rm Con}_h(Z)$.
Assume that $\sigma^{-1}(\sigma(z, t)) = (z_0, s)$. Then by \eqref{si-in-x} again,
$$
\sigma(z_0, s)=\sigma\circ \sigma^{-1}(\sigma(z, t))=\sigma(z, t).
$$
Moreover, by \eqref{sigma-h}, we know that
$$
s=t.
$$
Therefore, \eqref{y-y0-t} is true. In addition, if $v=\sigma(z, t)$  is a vertex of $X_Z$, since $\sigma(z, t)=\sigma(z_0, t)$, $\sigma(R_{z})=\gamma_{z}$ and $\sigma(R_{z_0})=\gamma_{z_0}$, it follows that $v\in \gamma_{z}\cap \gamma_{z_0}$. Moreover, by \eqref{sigma-h}, we have $t=\alpha_Z^{-h(v)}$.
As $\gamma_z$ $($ resp. $\gamma_{z_0}$$)$ is a vertical geodesic anchored at $z$ $($ resp. $z_0$$)$,   we have
	$$ z \in B_Z\left(\pi(v), \frac{1}{3}\tau_Z \alpha_Z^{-h(v)}\right) \ \ \text{ and } \ \ z_0\in B_Z\left(\pi(v),  \frac{1}{3}\tau_Z \alpha_Z^{-h(v)}\right).$$
Consequently,
$$
d_Z(z, z_0)\leq \frac{2}{3}\tau_Z \alpha_Z^{-h(v)}=\frac{2}{3}\tau_Z  t.
$$
Hence the statement $(ii)$ is true, and the proposition is proved.
\end{proof}

{\it In the rest of this section, we make the following assumptions. Let $(Z, d_Z)$ $($resp. $(W, d_W)$$)$ be a complete metric space. Let $X_Z$ $($resp. $X_W$$)$ be a hyperbolic filling of $Z$ $($resp. $W$$)$ with parameters $\alpha_Z>1$ and  $\tau_Z>\max\Big\{3, \frac{\alpha_Z}{\alpha_Z-1}\Big\}$ $\big($resp. $\alpha_W>1$ and  $\tau_W>\max\Big\{3, \frac{\alpha_W}{\alpha_W-1}\Big\}$$\big)$. Let  $V_Z$ $($resp. $V_W$$)$ denote the vertex set of $X_Z$ $($resp.  $X_W$$)$.
Let $\omega$ $($resp. $\omega^\prime$$)$ denote the point in $\partial_G X_Z$ $($resp. $\partial_G X_W$$)$ such that all anchored descending geodesic rays in $X_Z$ $($resp. $X_W$$)$ belong to it.
Let $\partial_\omega X_Z$ $($resp. $\partial_{\omega^\prime} X_W$$)$ be the Gromov boundary of $X_Z$ $($resp. $X_W$$)$ relative to $\omega$  $($resp. $\omega^\prime$$)$. Let $f: (Z, d_Z)\rightarrow (W, d_W)$ be a $(\theta, \lambda)$-power quasi-symmetric mapping with $\theta\geq 1$ and $\lambda\geq 1$.}




\medskip

Now, we are ready to give the rough quasi-isometric extension of  $f$.
\begin{thm}\label{thm-emb}
	There exists a mapping
$$F_e: X_Z\rightarrow X_W$$
such that it is an $(L_1, L_2, \Lambda)$-rough quasi-isometric mapping, where
$
L_1=\frac{\log\alpha_Z}{\theta\log\alpha_W},$ $L_2=\frac{\theta \log\alpha_Z}{\log\alpha_W}$ and $\Lambda=\Lambda(\alpha_Z, \tau_Z, \alpha_W, \tau_W, \theta, \lambda).
$
\end{thm}
\begin{proof}
By Proposition~\ref{thm-similar}, we see that there exists a mapping $\sigma: \text{Con}_h(Z)\to X_Z$ (resp. $\sigma^{\prime}: {\rm Con}_h(W) \to X_W$) which is a $(1/\log\alpha_Z, C_\sigma)$-rough similarity with $C_\sigma=C_\sigma(\alpha_Z, \tau_Z)$ (resp. a $(1/\log\alpha_W, C_{\sigma^\prime})$-rough similarity with $C_{\sigma^\prime}=C_{\sigma^\prime}(\alpha_W, \tau_W)$) and satisfies that for any $(\xi, s)\in \text{Con}_h(Z)$ (resp. for any $\big(\xi^\prime, s^\prime)\in \text{Con}_h(W)$\big),
$$
		\sigma(\xi, s)=\gamma_{\xi}\left(-\frac{\log s}{\log\alpha_Z}\right)\;\; \mbox{\Big(resp.}\;\; \sigma^\prime(\xi^\prime, s^\prime)=\gamma_{\xi^\prime}\left(-\frac{\log s^\prime}{\log\alpha_W}\right)\Big).
$$

 Also, we know from Proposition~\ref{inverse-sigma} that $\sigma$ has a rough inverse
	$$\sigma^{-1}: X_Z\rightarrow \text{Con}_h(Z),$$
which is a $(\log\alpha_Z, C_{\sigma^{-1}})$-rough similarity, where $C_{\sigma^{-1}}=C_{\sigma^{-1}}(\alpha_Z, \tau_Z)$. Let
	\begin{equation}\label{expre-F}
	F_e=\sigma^\prime \circ \widehat{f} \circ \sigma^{-1},
	\end{equation}
where $\widehat{f}: {\rm Con}_h(Z)\to {\rm Con}_h(W)$ is the $(\theta, k)$-rough quasi-isometric mapping induced by the power quasi-symmetric mapping $f$ as in Theorem~\ref{thm-cone} with $k=k(\theta,\lambda)$.
Then we see from Lemma~\ref{lem-composition} that
	$
	F_e: X_Z\rightarrow X_W
	$
	 is an $(L_1, L_2, \Lambda)$-rough quasi-isometric mapping, where
$L_1=\frac{\log\alpha_Z}{\theta\log\alpha_W},$ $L_2=\frac{\theta\log\alpha_Z}{\log\alpha_W}$ and $\Lambda=\frac{\theta(C_{\sigma^{-1}}+1)+2(k+1)}{\log\alpha_W}+2C_{\sigma^\prime}+1.$
This proves the theorem.
\end{proof}

As a corollary, we obtain an expression for the height of $F_e(x)$, where $x\in \gamma_z$ and $z\in Z$.

\begin{cor}\label{high-F}
 Suppose that $z\in Z$ and $x\in \gamma_z$. Then
$$
h(F_e(x))=\frac{\log2}{\log\alpha_W}\Phi_{z_0}\left(h(x)\log_2 \alpha_Z\right),
$$
where $z_0\in Z$ is the point satisfying
$$\sigma^{-1}(x)=\left(z_0, \alpha_Z^{-h(x)}\right).$$
\end{cor}
\begin{proof}
Let $x\in \gamma_z$. Then $x=\gamma_z(h(x))$ by \eqref{25-1-29-1}. Thus it follows from \eqref{f-con} that
$$
\sigma\left(z, \alpha_Z^{-h(x)}\right)=x.
$$
By Proposition~\ref{inverse-sigma}$(ii)$,  there exists a point $z_0 \in Z$ such that
$$
\sigma^{-1}(x)=\sigma^{-1}\circ\sigma\left(z, \alpha_Z^{-h(x)}\right)=\left(z_0, \alpha_Z^{-h(x)}\right).
$$

By the expression \eqref{expre-F}, together with \eqref{eq-def-fx} and \eqref{f-con}, we get
$$
F_e(x)=\gamma_{f(z_0)}\left(\frac{\log2}{\log\alpha_W}\Phi_{z_0}\left(h(x)\log_2 \alpha_Z\right)\right),
$$ from which the corollary follows.
\end{proof}



The following is another result concerning the mapping $F_e$ constructed in Theorem~\ref{thm-emb}. It  demonstrates that for any $z\in Z$, the image $F_e(\gamma_z)$ of a vertical geodesic $\gamma_z$ lies within a neighborhood of the vertical geodesic $\gamma_{f(z)}$. This result will play a crucial role in   showing that the boundary mapping of $F_e$ coincides with $f$ (see Theorem~\ref{prop-boundary}).

\begin{prop}\label{prop-close}
There exists a constant $\mu$ with $\mu=\mu(\alpha_Z, \tau_Z, \alpha_W, \tau_W, \theta, \lambda)$ such that for any $z\in Z$ and for any $x\in \gamma_z$,
	\begin{equation*}
		\dist(\{F_e(x)\}, \gamma_{f(z)})\leq \mu.
	\end{equation*}
\end{prop}
\begin{proof}
	Let $z\in Z$, and let $v\in \gamma_z\cap V_Z$. Since $\sigma$ maps $R_{z}$ onto $\gamma_{z}$ and since \eqref{25-1-29-1} implies $v=\gamma_z(h(v))$, we know from \eqref{sigma-h} that
$
\sigma(z, s)=v,
$
where
$
s=\alpha_Z^{-h(v)}.
$
By Proposition~\ref{inverse-sigma}$(ii)$, there exists $(z_0, s)\in {\rm Con}_h(Z)$ such that
\begin{equation}\label{eq-10-26}
\sigma(z_0, s)=v\in \gamma_{z_0} \ \ \text{ and } \ \ \sigma^{-1}(v)=(z_0, s).
\end{equation}	
Moreover,  as $v$ is a vertex, we also have
\begin{equation}\label{s-tauz}
d_Z(z, z_0)\leq \frac{2}{3}\tau_Z\alpha_Z^{-h(v)}=\frac{2}{3}\tau_Zs.
\end{equation}

\begin{claim}\label{25-6-3-1}
There is  a constant $C_1$ such that
	$$
		\left|\Phi_{z}\left(-\log_2 s\right)-\Phi_{z_0}\left(-\log_2 s\right)\right|\leq C_1,
$$
	where $C_1=C_1(\theta, \lambda, \tau_Z)$, and $\Phi_{z}$ and $\Phi_{z_0}$ are the functions defined in Subsection~\ref{sec-3-3}, which are $(\theta, \mu_0)$-rough quasi-isometric with $\mu_0=\mu_0(\theta, \lambda)$.
\end{claim}
	
	For the proof, by using the $(\theta, \mu_0)$-rough quasi-isometric properties of $\Phi_{z}$ and $\Phi_{z_0}$, we have
	\begin{equation}\label{10-9-1}
		\left|\Phi_{z}\left(-\log_2 s\right)-\Phi_{z}\left(\log_2\frac{3}{2\tau_Z s}\right)\right|\leq \theta\log_2\frac{2\tau_Z}{3}+\mu_0
	\end{equation}
	and
	\begin{equation*}
		\left|\Phi_{z_0}\left(-\log_2 s\right)-\Phi_{z_0}\left(\log_2\frac{3}{2\tau_Z s}\right)\right|\leq \theta\log_2\frac{2\tau_Z}{3}+\mu_0.
	\end{equation*}	
Moreover, Theorem~\ref{lem-cone-2} and \eqref{s-tauz} ensure that
\begin{equation*}
\left|\Phi_{z}\left(\log_2\frac{3}{2\tau_Z s}\right)-\Phi_{z_0}\left(\log_2\frac{3}{2\tau_Z s}\right)\right|\leq \mu_1,
\end{equation*} where $\mu_1=\mu_1(\theta, \lambda)$ is from Theorem~\ref{lem-cone-2}. Therefore, we obtain from the triangle inequality that
\begin{align*}
\left|\Phi_{z}\left(-\log_2 s\right)-\Phi_{z_0}\left(-\log_2 s\right)\right|
\leq C_1,
	\end{align*} where $C_1= 2\theta\log_2\frac{2\tau_Z}{3}+2\mu_0+\mu_1$.
This is what we need.

\begin{claim}\label{25-6-3-2}
There is a constant $C_2$ such that
	$$
		{\rho_h}(\widehat{f}(z, s), \widehat{f}(z_0, s))\leq C_2,
$$ where $C_2=C_2(\theta, \lambda, \tau_Z)$.
\end{claim}
	
For the proof, let $z^\prime=f(z)$ and $z_0^\prime=f(z_0)$.
	Since $f$ is a $(\theta, \lambda)$-power quasi-symmetric mapping, by \cite[Lemma 4.5]{HXu}, we know that there is a constant $C^{\prime}=C^{\prime}(\theta, \lambda)$ such that
	\begin{equation*}
		\left|\log_2\frac{1}{d_{W}(z^\prime, z_0^\prime)}-\Phi_{z}\left(\log_2\frac{1}{d_Z(z, z_0)}\right)\right|\leq C^{\prime}.
	\end{equation*}
It follows from \eqref{s-tauz} and \eqref{10-9-1} that
$$
d_{W}(z^\prime, z_0^\prime)\leq  2^{C^{\prime}-\Phi_{z}\left(-\log_2 d_Z(z, z_0) \right)}\leq 2^{C^{\prime}-\Phi_{z}\left(\log_2\frac{3}{2\tau_Zs}\right)}\leq 2^{C^{\prime}+\theta\log_2\frac{2\tau_Z}{3}+\mu_0 -\Phi_{z}\left(-\log_2 s\right)}.
$$
Then we get
\begin{equation*}
 d_{W}(z^\prime, z_0^\prime)+2^{-\Phi_{z}\left(-\log_2 s\right)}\vee2^{-\Phi_{z_0}\left(-\log_2 s\right)}
\leq 2^{C^{\prime}+\theta\log_2\frac{2\tau_Z}{3}+\mu_0+1}\left(2^{-\Phi_{z}\left(-\log_2 s\right)}\vee2^{-\Phi_{z_0}\left(-\log_2 s\right)}\right).
\end{equation*}

Since \eqref{metric-cone} and \eqref{eq-def-fx}  give
\begin{equation*}
{\rho_h}(\widehat{f}(z, s), \widehat{f}(z_0, s))
	 = 2\log\frac{d_{W}(z^\prime, z_0^\prime)+2^{-\Phi_{z}\left(-\log_2 s\right)}\vee2^{-\Phi_{z_0}\left(-\log_2 s\right)}}{\sqrt{2^{-\Phi_{z}\left(-\log_2 s\right)-\Phi_{z_0}\left(-\log_2 s\right)}}},	
\end{equation*}
it follows that
	\begin{align*}
		{\rho_h}(\widehat{f}(z, s), \widehat{f}(z_0, s))
		& \leq 2\log\frac{2^{-\Phi_{z}\left(-\log_2 s\right)}\vee2^{-\Phi_{z_0}\left(-\log_2 s\right)}}{\sqrt{2^{-\Phi_{z}\left(-\log_2 s\right)-\Phi_{z_0}\left(-\log_2 s\right)}}}+ 2\log 2^{C^{\prime}+\theta\log_2\frac{2\tau_Z}{3}+\mu_0+1} \\
		& =  \log 2 \cdot \left|\Phi_{z}\left(-\log_2 s\right)-\Phi_{z_0}\left(-\log_2 s\right)\right|+2\Big(C^{\prime}+\theta\log_2\frac{2\tau_Z}{3}+\mu_0+1\Big)\log 2.
	\end{align*}
Combining with Claim~\ref{25-6-3-1}, we have
	\begin{equation*}
		{\rho_h}(\widehat{f}(z, s), \widehat{f}(z_0, s))
  \leq C_2,
	\end{equation*} where $C_2=(C_1+2C^{\prime}+2\theta\log_2\frac{2\tau_Z}{3}+2\mu_0+2)\log 2$.
	This shows that the claim is true.\medskip
	

Observe that $\sigma^\prime$ maps the ray $R_{f(z)}$ onto the vertical geodesic
	$\gamma_{f(z)}$ anchorded at $f(z)$.
	This implies that $\sigma^\prime(\widehat{f}(z, s))\in\gamma_{f(z)}$.
Since $F_e=\sigma^\prime \circ \widehat{f} \circ \sigma^{-1}$ (see the proof of Theorem~\ref{thm-emb}), we know from \eqref{eq-10-26} that
$$\dist(\{F_e(v)\}, \gamma_{f(z)})\leq |\sigma^\prime(\widehat{f}(z_0, s))- \sigma^\prime(\widehat{f}(z, s))|.$$

Note that $\sigma^\prime$ is a $(1/\log\alpha_W, C_{\sigma^\prime})$-rough similarity with $C_{\sigma^\prime}=C_{\sigma^\prime}(\alpha_W, \tau_W)$) (see Proposition~\ref{thm-similar}). Then we infer from  Claim~\ref{25-6-3-2}  that
	$$
		\dist(\{F_e(v)\}, \gamma_{f(z)})
		\leq \frac{1}{\log\alpha_W}{\rho_h}(\widehat{f}(z, s), \widehat{f}(z_0, s))+C_{\sigma^\prime}
		\leq \frac{C_2}{\log\alpha_W} +C_{\sigma^\prime}.
	$$

To finish the proof,  let $x\in \gamma_z$. Then there exists a vertex $v\in\gamma_z\cap V_Z$ such that $|x-v|\leq 1$, and thus, it follows from Theorem~\ref{thm-emb} that
	$$
\dist(\{F_e(x)\}, \gamma_{f(z)})  \leq \dist(\{F_e(v)\}, \gamma_{f(z)})+|F_e(x)-F_e(v)| \leq \mu,
$$ where $\mu=\frac{C_2}{\log\alpha_W} +C_{\sigma^\prime}+L_2+\Lambda$.
\end{proof}

In the following, we show that the boundary behavior of the mapping $F_e$ constructed in Theorem~\ref{thm-emb} coincides with $f$.
Recall that $X_Z$ and $X_W$ are Gromov hyperbolic geodesic spaces.
By \cite[Proposition 6.3]{BSC}, we know that every rough quasi-isometric mapping $\Psi: X_Z\rightarrow X_W$ induces a boundary mapping
$$\partial_\infty\Psi: \partial _G X_Z\rightarrow \partial _G X_W,$$
which is defined as follows. If $\{x_n\}\subset X_Z$ converges to $x\in\partial _G X_Z$, then $\{\Psi(x_n)\}\subset X_W$ converges to infinity. Let $y$ be the equivalence class of $\{\Psi(x_n)\}$, and then, define $\partial_\infty\Psi(x)=y.$ We refer to \cite[Section 6]{BSC} for more discussions about $\partial_\infty \Psi$.
Moreover, combining with Theorem~\ref{prop-Gromov} and Theorem~\ref{prop-identi}, we see that $\partial_\infty\Psi$ is well-defined on $Z$ via the canonical identification.


\begin{thm}\label{prop-boundary}
Suppose that $F_e$ is the rough quasi-isometric extension of the quasi-symmetric mapping $f$ constructed in Theorem~\ref{thm-emb}. Then $\partial_\infty F_e= f$ on $Z$.
\end{thm}
\begin{proof}

Let $z\in Z$. Then \eqref{eq-6-24-1} implies that $\{\gamma_z(n)\}_{n\in \mathbb N}$ converges to $z$ with respect to $\omega$. By \eqref{f-con}, we know that
for each $n\in\mathbb N$,
$\sigma(z, s_n)=\gamma_z(n)$, where $s_n=\alpha_Z^{-n}$. Furthermore, Proposition~\ref{inverse-sigma}$(ii)$ ensures that there exists $(z_n, s_n)\in {\rm Con}_h(Z)$ such that
\begin{equation}\label{25-02-6-1}
\sigma^{-1}(\gamma_z(n))=(z_n, s_n) \ \ \text{ and } \ \ \sigma(z_n, s_n)=\sigma(z, s_n).
\end{equation}

It follows from \eqref{25-02-6-1} and Claim~\ref{25-6-3-1} in the proof of Proposition~\ref{prop-close} that
	\begin{equation*}
		\left|\Phi_{z}\left(-\log_2 s_n \right)-\Phi_{z_n}\left(-\log_2 s_n \right)\right|\leq C_1,
	\end{equation*}
where $C_1=C_1(\theta, \lambda, \tau_Z)$. We know  from Corollary~\ref{high-F} that
\begin{equation*}
h(F_e(\gamma_z(n)))
= \frac{\log2}{\log\alpha_W}\Phi_{z_n}\left(-\log_2 s_n \right)
\geq \frac{\log2}{\log\alpha_W}\Phi_{z}\left(-\log_2 s_n \right)-\frac{ \log2}{\log\alpha_W}C_1.
\end{equation*}
Combining with \eqref{mon-01} and the fact that $s_n=\alpha_Z^{-n}$, we obtain
\begin{equation}\label{25-02-6-2}
h(F_e(\gamma_z(n)))\to+\infty \ \text{ as } \  n\to+\infty.
\end{equation}

 Moreover, by Proposition~\ref{prop-close}, there exists a sequence $\{t_n\}_{n\in \mathbb N}\subset \mathbb Z$ such that for each $n\in \mathbb N$,
\begin{equation}\label{mmuu}
|F_e(\gamma_z(n))-\gamma_{f(z)}(t_n)|\leq \mu+1.
\end{equation}
Therefore, \eqref{eq-6-24} and the fact that the height function $h$ is $1$-Lipschitz ensure that
$$t_n=h(\gamma_{f(z)}(t_n))\geq h(F_e(\gamma_z(n)))-(\mu+1).$$
Then \eqref{25-02-6-2} implies
 $$t_n\to +\infty \ \text{ as } \   n\to +\infty.$$
 Hence we know from \eqref{eq-6-24-1} that $\{\gamma_{f(z)}(t_n)\}_{n\in \mathbb N}$ converges to $f(z)$ with respect to $\omega^\prime$. Since \eqref{mmuu} guarantees that  $\{F_e(\gamma_z(n))\}_{n\in \mathbb N}$ and $\{\gamma_{f(z)}(t_n)\}_{n\in \mathbb N}$ are equivalent with respect to $\omega^\prime$, we see that $\{F_e(\gamma_z(n))\}_{n\in \mathbb N}$ also converges to $f(z)$ with respect to $\omega^\prime$. This implies that
 $\partial_\infty F_e(z)= f(z)$. By the arbitrariness of $z$ in $Z$, we see that
 $\partial_\infty F_e=f$ on $Z$.
\end{proof}

In general, the rough quasi-isometric extension $F_e$ constructed in Theorem~\ref{thm-emb} may not map vertices of $X_Z$ to those of $X_W$.
However, the extension $F_e$ can be modified to satisfy the vertex-to-vertex property.

\begin{thm}\label{prop-trun-X}
Suppose that $F_e: X_Z\rightarrow X_W$ is the $(L_1, L_2, \Lambda)$ rough quasi-isometric extension constructed in Theorem~\ref{thm-emb}. Then there exists an $(L_1, L_2, \Lambda^\prime)$-rough quasi-isometric mapping $F: X_Z \to X_W$
with  $\Lambda^\prime=4L_2+5\Lambda+6$ such that
\ben
\item[$(1)$]
$F$ maps the vertex set $V_Z$ into the one $V_W$;

\item[$(2)$]
 for any $x\in X_Z$,
\begin{equation*}
|F_e(x)-F(x)|\leq \Theta,
\end{equation*} where $\Theta=2(L_2+\Lambda)+3$;

\item[$(3)$]
$\partial_\infty F= f$ on $Z$.
\een
\end{thm}
 \begin{proof}
We construct the required mapping $F$ in two steps.
In the first step, we construct a mapping from $V_Z$ to $V_W$ as follows. For $v\in V_Z$,  let
\begin{equation}\label{vertex-M}
V_*(v)=\{v_*:\; v_*\;\mbox{is a vertex in}\;V_W \; \mbox{closest to $F_e(v)$ with respect to the graph metric}\}.
\end{equation}

 Obviously, $V_*(v)\not=\emptyset$ for any $v\in V_Z$. For each $v\in V_Z$, we fix an element $v_*$ in $V_*(v)$, and then, define
$F(v)=v_*$. Clearly, this defines a mapping from $V_Z$ to $V_W$.

In the second step, we extend the above $F: V_Z \rightarrow V_W$ to a mapping from $X_Z$ to $X_W$, which is still denoted by $F$. For this, let $[v_1, v_2]$ denote an edge in $X_Z$, and let $[F(v_1), F(v_2)]$ be a geodesic segment in $X_W$ connecting $F(v_1)$ and $F(v_2)$.
For $x\in[v_1, v_2]$, there must exist $s\in[0, 1]$ such that
$$
|v_1-x|=s|v_1-v_2|,
$$
where $[v_1, x]\subset [v_1, v_2]$.
We define $F(x)$ to be the point in $[F(v_1), F(v_2)]$ such that
$$
|F(v_1)-F(x)|=s|F(v_1) - F(v_2)|,
$$
where $[F(v_1), F(x)]\subset [F(v_1), F(v_2)]$. In this way, we obtain a mapping $F$ from $X_Z$ to $X_W$.


Next, we show that $F$ satisfies the requirements in the theorem. It is clear that the statement $(1)$ of the theorem holds true. For the statement $(2)$, since for any $v\in V_Z$, by \eqref{vertex-M},
\begin{equation}\label{eq-star}
|F(v)-F_e(v)|\leq 1,
\end{equation}
we see that for any $v, w\in V_Z$,
\begin{equation}\label{rough-Fstar}
|F_e(v)-F_e(w)|-2\leq |F(v)-F(w)|\leq |F_e(v)-F_e(w)|+2.
\end{equation}

For any $x\in X_Z$, there exists an edge $[v, w]$ in $X_Z$ such that $x\in[v, w]$.
Since $F_e: X_{Z} \to X_{W}$ is an $(L_1, L_2, \Lambda )$-rough quasi-isometric mapping, we see that
\begin{equation}\label{25-1-26-1}
|F_e(x)-F_e(v)|\leq L_2+\Lambda.
\end{equation}
Then it follows from \eqref{rough-Fstar} and the construction of $F$ that
\begin{equation}\label{25-1-26-2}
|F(x)-F(v)|\leq |F(v)-F(w)|\leq |F_e(v)-F_e(w)|+2\leq  L_2+\Lambda +2.
\end{equation}

Since by the triangle inequality,
$$|F(x)-F_{e}(x)|\leq |F_e(x)-F_e(v)|+ |F_e(v)-F(v)|+|F(v)-F(x)|,$$
we infer from \eqref{eq-star}, \eqref{25-1-26-1} and \eqref{25-1-26-2} that
\begin{equation}\label{x-roughs}
|F(x)-F_{e}(x)|\leq 2(L_2+\Lambda )+3.
\end{equation}
Therefore, the statement $(2)$ of the theorem holds true.

In the following, we show that $F$ is an $(L_1, L_2, \Lambda^\prime)$-rough quasi-isometric mapping with  $\Lambda^\prime=4L_2+5\Lambda+6$.
Let $x_1, x_2\in X_Z$. We obtain from \eqref{x-roughs} that
\begin{align}
\Big| |F(x_1)-F(x_2)|- |F_e(x_1)-F_e(x_2)| \Big|&\leq |F(x_1)-F_e(x_1)|+|F_e(x_2)-F(x_2)|\notag\\
&\leq 4(L_2+\Lambda)+6.\label{F-F_e}
\end{align}
Since $F_e$ is $(L_1, L_2, \Lambda)$-rough quasi-isometric,
we obtain from \eqref{F-F_e} that
\begin{equation}\label{25-1-22-4}
L_1|x_1-x_2|-(4L_2+5\Lambda +6) \leq |F(x_1)-F(x_2)| \leq L_2|x_1-x_2|+4L_2+5\Lambda +6.
\end{equation}

For any $y\in X_W$, since $F_e(X_Z)$ is $\Lambda$-cobounded in $X_W$, we see that there is $x\in X_Z$ such that
\begin{equation*}
|F_e(x)-y|\leq \Lambda,
\end{equation*}
which, together with \eqref{x-roughs}, shows that
\begin{equation}\label{25-1-22-3}
|F(x)-y|\leq 2L_2+3\Lambda+3\leq \Lambda^\prime.
\end{equation}
Then it follows from \eqref{25-1-22-4} and \eqref{25-1-22-3} that $F: X_Z \to X_W$ is an \((L_1, L_2, \Lambda^\prime)\)-rough quasi-isometric mapping.

Finally,
by the statement $(2)$, it is clear that the sequences $\{F(x_n)\}_{n\in\mathbb N}$ and $\{F_e(x_n)\}_{n\in\mathbb N}$ are equivalent with respect to $\omega^\prime$ if $\{F_e(x_n)\}_{n\in\mathbb N}$ converges to infinity with respect to $\omega^\prime$ for $\{x_n\}_{n\in\mathbb N}\subset X_Z$. Thus it follows from Theorem~\ref{prop-boundary} that
 $\partial_\infty F=\partial_\infty F_e=f$ on $Z$, which establishes the statement $(3)$. The proof of this theorem is complete.
\end{proof}



\section{Boundary mappings of  rough quasi-isometric mappings}\label{sec-5}
The purpose of this section is to formulate and prove a yet general result from which Theorem~\ref{intro-thm-emb}$(ii)$ follows.

Let $(Z, d_Z)$ $($resp. $(W, d_W)$$)$ be a complete metric space, and let $X_Z$ $($resp. $X_W$$)$ a hyperbolic filling of $Z$ $($resp. $W$$)$ with parameters $\alpha_Z>1$, $\tau_Z>\max\Big\{3, \frac{\alpha_Z}{\alpha_Z-1}\Big\}$ $\big($resp. $\alpha_W>1$, $\tau_W>\max\Big\{3, \frac{\alpha_W}{\alpha_W-1}\Big\}$$\big)$. Denote by $V_Z$ $($resp.  $V_W$$)$ the vertex set of $X_Z$ $($resp.  $X_W$$)$.

Assume that $\omega$ $($resp. $\omega^\prime$$)$ denotes the unique point in $\partial_G X_Z$ $($resp. $\partial_G X_W$$)$ such that all anchored descending geodesic rays in $X_Z$ $($resp. $X_W$$)$ belong to it.
Let $\partial_\omega X_Z$ $($resp. $\partial_{\omega^\prime} X_W$$)$ be the Gromov boundary of $X_Z$ $($resp. $X_W$$)$ relative to $\omega$  $($resp. $\omega^\prime$$)$.
Let $d_\omega$ $($resp. $d_{\omega^\prime}$$)$  be  the visual metric on $\partial_\omega X_Z$  $($resp. $\partial_{\omega^\prime} X_W$$)$ with parameter $\epsilon=\log\alpha_Z$  $($resp. $\epsilon=\log\alpha_W$$)$, which is induced by the metric $d_Z$ $($resp. $d_W$$)$ as  in \eqref{d-visual}.

For any rough quasi-isometric mapping $\Psi: X_Z\rightarrow X_W$,  recall that $\Psi$ induces a boundary mapping
$$\partial_\infty\Psi: \partial _G X_Z\rightarrow \partial _G X_W,$$
which is defined by $\partial_\infty\Psi(x)=y$, where $y$ is the equivalence class of $\{\Psi(x_n)\}$ for some sequence $\{x_n\}\in x$.

\begin{thm}\label{thm-limit}
	Suppose that $\Psi: X_Z\rightarrow X_W$ is an $(L_1, L_2, \Lambda)$-rough quasi-isometric mapping with $L_2\geq L_1>0$ and $\Lambda\geq0$.
	If the boundary mapping $\partial_\infty \Psi$ maps $\omega$ to ${\omega^\prime}$, then $\partial_\infty \Psi:$ $(Z, d_Z)\to
	(W, d_W)$ is an $\eta$-quasi-symmetric mapping, where
	\begin{equation*}
		\eta(t)=
		\left\{\begin{array}{cl}
			\lambda t^{\theta_1},& {\rm for} \;\; 0<t<1, \\
			\lambda t^{\theta_2},& {\rm for} \;\; t\geq1,
		\end{array}\right.
	\end{equation*}
	$\theta_1=\frac{\log\alpha_W}{\log\alpha_Z}L_1$, $\theta_2=\frac{\log\alpha_W}{\log\alpha_Z}L_2$ and
	$\lambda=\lambda(\alpha_Z, \tau_Z, \alpha_W, \tau_W, L_1, L_2, \Lambda)$.
\end{thm}
\begin{proof}
	By Theorem~\ref{prop-identi} and \eqref{d-visual}, there exists a canonical identification between $(Z, d_Z)$ (resp. $(W, d_W)$) and $(\partial_\omega X_Z, d_\omega)$ (resp. $(\partial_{\omega^\prime} X_W, d_{\omega^\prime})$), which is indeed an isometry. Therefore, to prove the theorem, it suffices  to show that $\partial_\infty \Psi$ is an $\eta$-quasi-symmetric mapping from $(\partial_\omega X_Z, d_\omega)$ to $(\partial_{\omega^\prime} X_W, d_{\omega^\prime})$.

	Assume that $o\in X$ and $o^\prime=\Psi(o)$.  Let $\gamma$ (resp. $\gamma^\prime$) be a fixed geodesic ray from $o$ (resp. $o^\prime$) to $\omega$ (resp. $\omega^\prime$) with $\gamma(0)=o$ (resp. $\gamma^\prime(0)=o^\prime$), and let $b$ (resp. $b^\prime$) be the Busemann function associated to $\gamma$ (resp. $\gamma^\prime$) based at $\omega$ (resp. $\omega^\prime$).
	Recall that $X_Z$ (resp. $X_W$) is a $\delta_Z$-hyperbolic (resp. $\delta_W$-hyperbolic) geodesic space with $\delta_Z=\delta_Z(\alpha_Z, \tau_Z)$ $($resp. $\delta_W=\delta_W(\alpha_W, \tau_W)$$)$. We start the proof of the theorem with two claims.

	
	\begin{claim}\label{25-6-5-1}
		There exists a constant $C_\Psi\geq 1$ such that for any distinct points $x, y, z\in \partial_\omega X_Z$,
		\begin{equation*}\label{ex-claim-1}
			\alpha_W^{\langle \partial_\infty \Psi(x), \partial_\infty \Psi(y), \partial_\infty \Psi(z), \partial_\infty \Psi(\omega)\rangle}\leq
			\left\{\begin{array}{cl}
				C_\Psi\alpha_W^{L_2\langle x, y, z, \omega\rangle},& {\rm if} \;\; \langle x, y, z, \omega\rangle\geq0, \\
				C_\Psi\alpha_W^{L_1\langle x, y, z, \omega\rangle},& {\rm if} \;\; \langle x, y, z, \omega\rangle<0,
			\end{array}\right.
		\end{equation*}
		where $C_\Psi=C_\Psi(\alpha_Z, \tau_Z, \alpha_W, \tau_W, L_1, L_2, \Lambda)$.
	\end{claim}
	
	To prove the claim, we need some preparation.
	By Lemma~\ref{thm-strong}, $\Psi: X_Z\to X_W$ is a strongly $(L_1, L_2, \Lambda^\prime)$-power quasi-isometric mapping with $\Lambda^\prime=\Lambda^\prime(\alpha_Z, \tau_Z, \alpha_W, \tau_W, L_1, L_2, \Lambda)$. This means that for any quadruple of points $\{u_1, u_2, u_3, u_4\}$ in $X_Z$, if
	$\langle u_1, u_2, u_3, u_4\rangle\geq 0$, then
	\begin{equation}\label{strong-XZXW}
		L_1\langle u_1, u_2, u_3, u_4\rangle-\Lambda^\prime\leq \langle \Psi(u_1), \Psi(u_2), \Psi(u_3), \Psi(u_4)\rangle\leq L_2\langle u_1, u_2, u_3, u_4\rangle+\Lambda^\prime.
	\end{equation}
	Moreover, if  $\langle u_1, u_2, u_3, u_4\rangle< 0$, then \eqref{eq-strong} gives
	\begin{equation}\label{25-6-7-1}
		L_2\langle u_1, u_2, u_3, u_4\rangle-\Lambda^\prime\leq \langle \Psi(u_1), \Psi(u_2), \Psi(u_3), \Psi(u_4)\rangle\leq L_1\langle u_1, u_2, u_3, u_4\rangle+\Lambda^\prime.
	\end{equation}
	
	Since $0<L_1\leq L_2$, we conclude from \eqref{strong-XZXW} and \eqref{25-6-7-1} that for any $u_1$, $u_2$, $u_3$, $u_4\in X_Z$,
	\begin{equation}\label{25-02-11-1}
		-L_2\big|\langle u_1, u_2, u_3, u_4\rangle\big|-\Lambda^\prime\leq \langle \Psi(u_1), \Psi(u_2), \Psi(u_3), \Psi(u_4)\rangle\leq L_2\big|\langle u_1, u_2, u_3, u_4\rangle\big|+\Lambda^\prime.
	\end{equation}
	
	Now, we are ready to prove  Claim~\ref{25-6-5-1}. Let $x$, $y$ and $z$ be distinct points in $\partial_\omega X_Z$, and let $\{x_n\}\in x$, $\{y_n\}\in y$, $\{z_n\}\in z$ and $\{\omega_n\}\in \omega$ be sequences in $X_Z$. Then Lemma~\ref{lem-langle}$(i)$ ensures that there exists  a constant $C_1=C_1(\alpha_Z, \tau_Z)\geq 0$ such that
	\begin{align*}
		\langle x, y, z, \omega\rangle-C_1 \leq \liminf_{n\to\infty}\langle x_n, y_n, z_n, \omega_n\rangle
		\leq \langle x, y, z, \omega\rangle+C_1.
	\end{align*}

	Since $\{x_n\}\in x$ implies that $\{\Psi(x_n)\}\in \partial_\infty \Psi(x)$,  again, by Lemma~\ref{lem-langle}$(i)$, we know that there exists a constant $C_1^\prime=C_1^\prime(\alpha_W, \tau_W)\geq 0$ satisfying
	\begin{align}\label{25-02-11-2}
		\langle \partial_\infty \Psi(x), \partial_\infty \Psi(y), \partial_\infty \Psi(z), \partial_\infty \Psi(\omega)\rangle-C_1^\prime
		&\leq  \liminf_{n\to\infty}\langle \Psi(x_n),  \Psi(y_n),  \Psi(z_n),  \Psi(\omega_n)\rangle \notag\\
		&\leq \langle \partial_\infty \Psi(x), \partial_\infty \Psi(y), \partial_\infty \Psi(z), \partial_\infty \Psi(\omega)\rangle+C_1^\prime.
	\end{align}
	
	Moreover, we need a relation between the quantities $\langle \partial_\infty \Psi(x),$ $\partial_\infty \Psi(y), \partial_\infty \Psi(z), \partial_\infty \Psi(\omega)\rangle$ and $\langle x, y, z, \omega\rangle$, which is formulated in \eqref{Langle} below. To reach this goal, we divide the arguments into the following two cases.
	
\bca Suppose that $\langle x, y, z, \omega\rangle>C_1$.
\eca
Since this assumption implies $\liminf_{n\to\infty}\langle x_n, y_n, z_n, \omega_n\rangle> 0,$ we see that  $\langle x_n, y_n, z_n, \omega_n\rangle\geq 0$ for sufficiently large $n$. It follows from \eqref{strong-XZXW}, \eqref{25-02-11-2} and Lemma~\ref{lem-langle}$(i)$ that
	\begin{align}\label{25-02-11-3}
		\langle \partial_\infty \Psi(x), \partial_\infty \Psi(y), \partial_\infty \Psi(z), \partial_\infty \Psi(\omega)\rangle
		&\leq L_2\liminf_{n\to\infty}\langle x_n, y_n, z_n, \omega_n\rangle+\Lambda^\prime+C_1^\prime \notag \\
		&\leq L_2\langle x, y, z, \omega\rangle+L_2 C_1+\Lambda^\prime+C_1^\prime
	\end{align}
	and
	\begin{align}\label{25-02-11-4}
		\langle \partial_\infty \Psi(x), \partial_\infty \Psi(y), \partial_\infty \Psi(z), \partial_\infty \Psi(\omega)\rangle
		&\geq  L_1\liminf_{n\to\infty}\langle x_n, y_n, z_n, \omega_n\rangle-\Lambda^\prime-C_1^\prime \notag \\
		&\geq  L_1\langle x, y, z, \omega\rangle-L_1 C_1-\Lambda^\prime-C_1^\prime.
	\end{align}
	
\bca	Suppose that $0\leq \langle x, y, z, \omega\rangle\leq C_1$.
\eca

 Since Lemma~\ref{lem-langle}$(i)$ gives
	$$
	\left|\liminf_{n\to\infty}\langle x_n, y_n, z_n, \omega_n\rangle\right|\leq  \langle x, y, z, \omega\rangle+C_1\leq 2C_1,
	$$
	we deduce from \eqref{25-02-11-1} and \eqref{25-02-11-2} that
	\begin{align}\label{25-02-11-5}
		\langle \partial_\infty \Psi(x), \partial_\infty \Psi(y), \partial_\infty \Psi(z), \partial_\infty \Psi(\omega)\rangle
		&\leq L_2\left|\liminf_{n\to\infty}\langle x_n, y_n, z_n, \omega_n\rangle\right|+\Lambda^\prime+C_1^\prime \notag\\
		&\leq L_2\langle x, y, z, \omega\rangle+ L_2C_1+\Lambda^\prime+C_1^\prime
	\end{align}
	and
	\begin{align}\label{25-02-11-6}
		\langle \partial_\infty \Psi(x), \partial_\infty \Psi(y), \partial_\infty \Psi(z), \partial_\infty \Psi(\omega)\rangle
		&\geq -L_2\left|\liminf_{n\to\infty}\langle x_n, y_n, z_n, \omega_n\rangle\right|-\Lambda^\prime-C_1^\prime \notag\\
		&\geq -2L_2C_1-\Lambda^\prime-C_1^\prime \notag\\
		&\geq L_1\langle x, y, z, \omega\rangle-3L_2C_1-\Lambda^\prime-C_1^\prime,
	\end{align}
	where in the last inequality, the assumption that  $0\leq \langle x, y, z, \omega\rangle\leq C_1$ and the fact of $0<L_1\leq L_2 $ are applied.
	
	By the relations \eqref{25-02-11-3}$-$\eqref{25-02-11-6}, we conclude that for any distinct points $x$, $y$ and $z$ in $\partial_\omega X_Z$, if $\langle x, y, z, \omega\rangle\geq 0$, then
	\begin{equation}\label{Langle}
		L_1\langle x, y, z, \omega\rangle-\Lambda^{\prime\prime}\leq \langle \partial_\infty \Psi(x), \partial_\infty \Psi(y), \partial_\infty \Psi(z), \partial_\infty \Psi(\omega)\rangle\leq L_2\langle x, y, z, \omega\rangle+\Lambda^{\prime\prime},
	\end{equation}
	where $\Lambda^{\prime\prime}=3L_2C_1+\Lambda^\prime+C_1^\prime$.
	
 Now, we are ready to finish the proof of the claim.	If $\langle x, y, z, \omega\rangle\geq0$, then it follows from \eqref{Langle} that
	\begin{equation*}
		\alpha_W^{\langle \partial_\infty \Psi(x), \partial_\infty \Psi(y), \partial_\infty \Psi(z), \partial_\infty \Psi(\omega)\rangle}
		\leq \alpha_W^{L_2\langle x, y, z, \omega\rangle+\Lambda^{\prime\prime}}.
	\end{equation*}
	If $\langle x, y, z, \omega\rangle<0$, then \eqref{eq-strong} implies $\langle x, z, y, \omega\rangle>0$, and thus, \eqref{Langle} gives
	\begin{align*}
		\alpha_W^{\langle \partial_\infty \Psi(x), \partial_\infty \Psi(y), \partial_\infty \Psi(z), \partial_\infty \Psi(\omega)\rangle}
		=\alpha_W^{-\langle \partial_\infty \Psi(x), \partial_\infty \Psi(z), \partial_\infty \Psi(y), \partial_\infty \Psi(\omega)\rangle}
		\leq \alpha_W^{L_1\langle x, y, z, \omega\rangle+\Lambda^{\prime\prime}}.
	\end{align*}
	By setting $C_\Psi=\alpha_W^{\Lambda^{\prime\prime}}$, we see that Claim~\ref{25-6-5-1} holds true.
	
	\smallskip
	\begin{claim}\label{25-6-5-2} There exist constants $C_Z=C_Z(\alpha_Z, \tau_Z)\geq 1$ and $C_W=C_W(\alpha_W, \tau_W)\geq 1$ such that for any distinct points $x, y, z\in \partial_{\omega}X_Z$,
		\begin{equation}\label{visu-1}
			C_Z^{-1}\alpha_Z^{\langle x, y, z, \omega\rangle} \leq \frac{d_\omega(x, z)}{d_\omega(x, y)}\leq C_Z\alpha_Z^{\langle x, y, z, \omega\rangle},
		\end{equation}
		and for any distinct points $x^\prime, y^\prime, z^\prime\in \partial_{\omega^\prime} X_W$,
		\begin{equation}\label{visu-2}
			C_W^{-1}\alpha_W^{\langle x^\prime, y^\prime, z^\prime, \omega^\prime\rangle} \leq \frac{d_{\omega^\prime}(x^\prime, z^\prime)}{d_{\omega^\prime}(x^\prime, y^\prime)}\leq C_W\alpha_W^{\langle x^\prime, y^\prime, z^\prime, \omega^\prime\rangle}.
		\end{equation}
		
	\end{claim}
	
	To prove the claim, let $x$, $y$ and $z$ be three distinct points in $ \partial_{\omega}X_Z$. By \eqref{compare}, there exist constants $C_2=C_2(\alpha_Z, \tau_Z)\geq 1$ and $C'_2=C'_2(\alpha_W, \tau_W)\geq 1$ such that
	\begin{align}\label{25-02-11-7}
		C_2^{-1}\alpha_Z^{(x|y)_b-(x|z)_b}\leq \frac{d_\omega(x, z)}{d_\omega(x, y)}\leq C_2\alpha_Z^{(x|y)_b-(x|z)_b}
	\end{align}
	and
	\begin{align}\label{25-02-11-8}
		(C'_2)^{-1}\alpha_W^{(x^\prime|y^\prime)_b-(x^\prime|z^\prime)_b}\leq \frac{d_{\omega^\prime}(x^\prime, z^\prime)}{d_{\omega^\prime}(x^\prime, y^\prime)}\leq C'_2\alpha_W^{(x^\prime|y^\prime)_b-(x^\prime|z^\prime)_b}.
	\end{align}
	
	Also, we see from Lemma~\ref{lem-langle}$(ii)$ that there exist constants $C_3=C_3(\alpha_Z, \tau_Z)\geq 0$ and $C_3^\prime=C_3^\prime(\alpha_W, \tau_W)\geq 0$ such that
	\begin{equation}\label{25-02-11-9}
		\big|(x|y)_b-(x|z)_b-\langle x, y, z, \omega\rangle\big|\leq C_3
	\end{equation}
	and
	\begin{equation}\label{25-02-11-10}
		\big|(x^\prime|y^\prime)_b-(x^\prime|z^\prime)_b-\langle x^\prime, y^\prime, z^\prime, \omega^\prime\rangle\big|\leq C_3^\prime.
	\end{equation}

	By setting $C_Z=C_2\alpha_Z^{C_3}$ and $C_W=C'_2\alpha_W^{C_3^\prime}$, it is evident that \eqref{visu-1}  follows from \eqref{25-02-11-7} and \eqref{25-02-11-9}, and \eqref{visu-2}  follows from \eqref{25-02-11-8} and \eqref{25-02-11-10}. Therefore Claim~\ref{25-6-5-2} is proved.\medskip
	
	Now, we are ready to prove the theorem based on Claims~\ref{25-6-5-1} and \ref{25-6-5-2}.
	Since $\Psi: X_Z\rightarrow X_W$ is a rough quasi-isometric mapping, by \cite[Proposition 6.3$(4)$]{BSC}, we know that $\partial_\infty \Psi: \partial_G X_Z\to \partial_G X_W$ is a bijection. Recall that $\partial_{\omega}X_Z=\partial_G X_Z\setminus\{\omega\}$ and $\partial_{\omega^\prime}X_W=\partial_G X_W\setminus\{\omega^\prime\}$ (see Theorem~\ref{prop-Gromov}).
	Then the assumption of $\partial_\infty \Psi(\omega)=\omega^\prime$ implies that $\partial_\infty \Psi$ is also a bijiection between $\partial_{\omega}X_Z$ and $\partial_{\omega^\prime}X_W$.
	This means that $x\in\partial_{\omega}X_Z$ if and only if $\partial_\infty \Psi(x)\in\partial_{\omega^\prime}X_W$.
	
	Let $x, y, z\in \partial_{\omega}X_Z$ be three distinct points. Then $\partial_\infty \Psi(x)$, $\partial_\infty \Psi(y)$ and $\partial_\infty \Psi(z)$ are distinct points in
	$\partial_{\omega^\prime} X_W$.
	To construct the needed control function $\eta$ in the theorem, we divide the discussions into the following three cases.
	
	If $d_\omega(x, z)> C_Z d_\omega(x, y)$, then we infer from \eqref{visu-1} in Claim~\ref{25-6-5-2} that $\langle x, y, z, \omega\rangle>0$, and thus, it follows from Claims~\ref{25-6-5-1} and \ref{25-6-5-2} that
	\begin{equation}\label{qsm-1}
		\frac{d_{\omega^\prime}(\partial_\infty \Psi(x), \partial_\infty \Psi(z))}{d_{\omega^\prime}(\partial_\infty \Psi(x), \partial_\infty \Psi(y))}\leq \lambda_1 \left(\frac{d_\omega(x, z)}{d_\omega(x, y)}\right)^{\frac{\log\alpha_W}{\log\alpha_Z}L_2},
	\end{equation}
	where $\lambda_1=C_WC_\Psi C_Z^{\frac{\log \alpha_W}{\log\alpha_Z}L_2}$.
	
	If $d_\omega(x, z)< C_Z^{-1} d_\omega(x, y)$, then \eqref{visu-1} in Claim~\ref{25-6-5-2} ensures that $\langle x, y, z, \omega\rangle<0$, and thus, it follows from  Claims~\ref{25-6-5-1} and~\ref{25-6-5-2} that
	\begin{equation}\label{qsm-2}
		\frac{d_{\omega^\prime}(\partial_\infty \Psi(x), \partial_\infty \Psi(z))}{d_{\omega^\prime}(\partial_\infty \Psi(x), \partial_\infty \Psi(y))}\leq \lambda_2 \left(\frac{d_\omega(x, z)}{d_\omega(x, y)}\right)^{\frac{\log\alpha_W}{\log\alpha_Z}L_1},
	\end{equation}  where $\lambda_2=C_WC_\Psi C_Z^{\frac{\log \alpha_W}{\log\alpha_Z}L_1}$.
	
For the remaining case, that is, $C_Z^{-1}d_\omega(x, y)\leq d_\omega(x, z)\leq C_Z d_\omega(x, y)$, since
$$C_Z^{-1}\leq \frac{d_\omega(x, y)}{d_\omega(x, z)}\leq C_Z,$$
we deduce from \eqref{visu-1}  in Claim~\ref{25-6-5-2} that
	$$
	C_Z^{-2} \leq \alpha_Z^{\langle x, y, z, \omega\rangle} \leq C_Z^2.
	$$
Therefore, by invoking Claim~\ref{25-6-5-1} and \eqref{visu-2} in Claim~\ref{25-6-5-2}, we derive the following estimate:
	\begin{align*}
		\frac{d_{\omega^\prime}(\partial_\infty \Psi(x), \partial_\infty \Psi(z))}{d_{\omega^\prime}(\partial_\infty \Psi(x), \partial_\infty \Psi(y))}
		&\leq C_W\alpha_W^{\langle \partial_\infty \Psi(x), \partial_\infty \Psi(y), \partial_\infty \Psi(z), \partial_\infty \Psi(\omega)\rangle}\\
		&\leq C_WC_\Psi \alpha_Z^{\max\big\{\frac{\log\alpha_W}{\log\alpha_Z}L_1\langle x, y, z, \omega\rangle, \frac{\log\alpha_W}{\log\alpha_Z}L_2\langle x, y, z, \omega\rangle\big\}}
		\leq C_I,
	\end{align*} where $C_I=C_WC_\Psi  C_Z^{\frac{2\log\alpha_W}{\log\alpha_Z}L_2}$.
	
Consequently,	if $C_Z^{-1} d_\omega(x, y)\leq d_\omega(x, z)\leq d_\omega(x, y)$, we have
	\begin{equation}\label{qsm-3}
		\frac{d_{\omega^\prime}(\partial_\infty \Psi(x), \partial_\infty \Psi(z))}{d_{\omega^\prime}(\partial_\infty \Psi(x), \partial_\infty \Psi(y))}
		\leq C_I \left(C_Z\frac{d_\omega(x, z)}{d_\omega(x, y)}\right)^{\frac{\log\alpha_W}{\log\alpha_Z}L_1}=\lambda_3\left(\frac{d_\omega(x, z)}{d_\omega(x, y)}\right)^{\frac{\log\alpha_W}{\log\alpha_Z}L_1},
	\end{equation}  where $\lambda_3=C_I C_Z^{\frac{\log\alpha_W}{\log\alpha_Z}L_1}$.
		If $d_\omega(x, y)\leq d_\omega(x, z)\leq C_Z d_\omega(x, y)$, we obtain
	\begin{equation}\label{qsm-4}
		\frac{d_{\omega^\prime}(\partial_\infty \Psi(x), \partial_\infty \Psi(z))}{d_{\omega^\prime}(\partial_\infty \Psi(x), \partial_\infty \Psi(y))}
		\leq \lambda_4\left(\frac{d_\omega(x, z)}{d_\omega(x, y)}\right)^{\frac{\log\alpha_W}{\log\alpha_Z}L_2},
	\end{equation}  where $\lambda_4=C_I$.
	
	By combining \eqref{qsm-1}$-$\eqref{qsm-4}, we see that the boundary mapping
	$\partial_\infty \Psi: (\partial_\omega X_Z, d_\omega)\rightarrow (\partial_{\omega^\prime} X_W, d_{\omega^\prime})$
	is an $\eta$-quasi-symmetric mapping with
	\begin{equation*}
		\eta(t)=
		\left\{\begin{array}{cl}
			\lambda t^{\theta_1},& \text{for} \;\; 0<t<1, \\
			\lambda t^{\theta_2},& \text{for} \;\; t\geq1,
		\end{array}\right.
	\end{equation*}
	where $\theta_1=\frac{\log\alpha_W}{\log\alpha_Z}L_1$, $\theta_2=\frac{\log\alpha_W}{\log\alpha_Z}L_2$ and $\lambda=\max\{\lambda_1, \lambda_2, \lambda_3, \lambda_4\}\leq C_WC_\Psi  C_Z^{\frac{3\log\alpha_W}{\log\alpha_Z}L_2}$.
\end{proof}

\begin{rem}\rm
	Suppose that $f: Z\rightarrow W$ is a $(\theta, \lambda)$-power quasi-symmetric mapping between two complete metric spaces $(Z, d_Z)$ and $(W, d_W)$ with $\theta\geq 1$ and $\lambda\geq 1$. By Theorem~\ref{thm-emb}, we
	see that $f$ can be extended to a rough quasi-isometric mapping $F: X_Z\rightarrow X_W$, and then, by applying Theorem~\ref{thm-limit}, we get $f$ back with the same exponents. This shows that the parameters $L_1$ and $L_2$ in Theorem~\ref{thm-emb} as well as the exponents $\theta_1$ and $\theta_2$ in Theorem~\ref{thm-limit} are all sharp.
\end{rem}

	\section*{Acknowledgments}
	M. Huang was partially supported by the National Natural Science Foundation of China (NSFC) under Grant No.~12371071. X. Wang and Z. Xu were partially supported by NSFC under Grant No.~12571081. Z. Wang was partially supported by the Natural Science Foundation of Hunan Province under Grant No.~2024JJ6299 and by the NSFC under Grants No.~12101226 and~12371071.

\bigskip

\noindent Manzi Huang,

\noindent{\it E-mail}:  \texttt{mzhuang@hunnu.edu.cn}

\medskip

\noindent Xiantao Wang,

\noindent{\it E-mail}:  \texttt{xtwang@hunnu.edu.cn}

\medskip

\noindent Zhuang Wang,

\noindent{\it E-mail}:  \texttt{zwang@hunnu.edu.cn}

\medskip

\noindent
Key Laboratory of Computing and Stochastic Mathematics (Ministry of Education), School of Mathematics and Statistics, Hunan Normal University, Changsha, Hunan 410081, P. R. China.

\bigskip

\noindent Zhihao Xu,

\noindent{\it E-mail}:  \texttt{734669860@qq.com}
\medskip

\noindent College of Mathematics and Statistics, Hengyang Normal University, Hengyang, Hunan 421010,  P. R. China.

\end{document}